\documentclass[12pt]{amsart}

\newif\ifdebug                                                      %
\debugfalse

\newcommand{\printname}[1]
   {\smash{\makebox[0pt]{\hspace{-1.0in}\raisebox{8pt}{\tiny #1}}}}

\newcommand{\labell}[1] {\label{#1}{\ifdebug{\printname{#1}}\fi}}

\usepackage{amssymb}
\usepackage{amsmath}
\usepackage{amscd}
\usepackage{wasysym}

\usepackage[latin1]{inputenc}
\usepackage{epic}
\usepackage{eepic}
\usepackage[all]{xy}

\def \v {{\lambda;\delta_1,\ldots,\delta_k}}
\def \R  {{\mathbb R}}
\def \Z  {{\mathbb Z}}

\def \N  {{\mathbb N}}
\def \CP {{{\mathbb C}{\mathbb P}}}

\def \CP {{\mathbb C}{\mathbb P}}

\def \calH {{\mathcal H}}
\def \calE {{\mathcal E}}

\def    \FS  {\text{FS}}

\def    \eps    {\epsilon}

\def    \ol {\overline}

\def\eoe{\unskip\ \hglue0mm\hfill$\between$\smallskip\goodbreak}

\DeclareMathOperator \cremona {cremona}
\DeclareMathOperator \defect {defect}

\DeclareMathOperator \GW {GW}

\DeclareMathOperator \red {red}

\numberwithin{equation}{section}
\swapnumbers

\newtheorem{Lemma}[equation]{Lemma}
\newtheorem{Theorem}[equation]{Theorem}
\newtheorem*{thm*}{Theorem}

\newtheorem{Proposition}[equation]{Proposition}
\newtheorem{Corollary}[equation]{Corollary}

\newtheorem*{Lemma*}{Lemma}
\newtheorem*{Corollary*}{Corollary}

\theoremstyle{remark}
\newtheorem{Remark}[equation]{Remark}
\newtheorem*{Remark*}{Remark}
\newtheorem{Example}[equation]{Example}

\theoremstyle{definition}

\newtheorem{Definition}[equation]{Definition}
\newtheorem{noTitle}[equation]{}

\begin{document}

\title{Distinguishing symplectic blowups of the complex projective plane}


\author{Yael Karshon}
\address{Department of Mathematics, University of Toronto,
40 St.\ George Street, Toronto Ontario M5S 2E4, Canada}
\email{karshon@math.toronto.edu}

\author{Liat Kessler}
\address{Department of Mathematics, Physics, and Computer Science, 
University of Haifa at Oranim, Tivon 36006, Israel}
\email{lkessler@math.haifa.ac.il}

\thanks{The first author is partially supported by 
the Natural Science and Engineering Research Council of Canada.
The second author was partially supported by the Center for Absorption in
Science, Ministry of Immigrant Absorption, State of Israel,
and by the Israel Science Foundation, Grant 557/08.}

\begin{abstract}
A symplectic manifold that is obtained from $\CP^2$ by $k$ blowups
is encoded by $k+1$ parameters: the size of the initial $\CP^2$,
and the sizes of the blowups.  We determine which values of these
parameters yield symplectomorphic manifolds.
\end{abstract}

\maketitle

\tableofcontents

\section{Introduction}
\labell{sec:intro}

A symplectic manifold that is obtained from $\CP^2$ by $k$ blowups 
is encoded by $k+1$ parameters: the size $\lambda$ of the initial $\CP^2$,
and the sizes $\delta_1,\ldots,\delta_k$ of the blowups.
In this paper we answer the following question:
\begin{quotation}
Which values of the parameters yield symplectomorphic manifolds?
\end{quotation}

{
\begin{Example} \labell{different blowups} \
For each of the vectors $(\lambda;\delta_1,\delta_2,\delta_3)$ 
in the table below,
consider the manifold that is obtained from a $\CP^2$ of size $\lambda$
by blowups of sizes $\delta_1$, $\delta_2$, $\delta_3$.
These three manifolds have the same classical invariants:
the symplectic volume, which is proportional to
${\lambda}^2-\sum_{j=1}^{3}{\delta_{j}}^2$;
the pairing of the symplectic form with the first Chern class,
which is proportional to
$3 \lambda-\sum_{j=1}^{3}\delta_{j}$;
and the set of values that the symplectic form takes on $H_2(M)$,
which is proportional to $\Z \lambda + \Z \delta_1 + \Z \delta_2 + \Z \delta_3$.
The first two manifolds
are symplectomorphic, but the third is not symplectomorphic to the first two.
\smallskip
\begin{center}
{\renewcommand{\arraystretch}{1.2}
 \renewcommand{\tabcolsep}{2mm}
\begin{tabular}{|cccc|}
\multicolumn{1}{c}{$\lambda$} & $\delta_1$ & $\delta_2$ 
& \multicolumn{1}{c}{$\delta_3$} \\ 

\hline
 15 & 9 & 5 & 4 \\
\hline
 12 & 6 & 2 & 1 \\
\hline
 11 & 4 & 1 & 1 \\
\hline 
\end{tabular} }
\end{center}
\smallskip
\end{Example}
}

\bigskip

\begin{Definition} \labell{reduced form}
Let $k \geq 3$, and let $\lambda, \delta_1, \ldots, \delta_k$ be real numbers.
The vector $(\lambda ; \delta_1 , \ldots , \delta_k)$
is \textbf{reduced} if
\begin{equation} \labell{conditions-1}
 \delta_1 \geq \ldots \geq \delta_k \quad \text{ and } \quad
   \delta_1 + \delta_2 + \delta_3 \leq \lambda.
\end{equation}
\end{Definition}

Our convention is that the \textbf{size} of $\CP^2$
equipped with a symplectic form
is $1/2\pi$ times the symplectic area of a line $\CP^1 \subset \CP^2$
and the \textbf{size} of a blowup is $1/2\pi$ times the symplectic area
of the exceptional divisor.
We normalize the Fubini-Study form $\omega_\FS$ so that it has size one.
We denote by 
$$ (M_k,\omega_{\lambda;\delta_1,\ldots,\delta_k}) $$
a symplectic manifold that is obtained from
$(\CP^2,\lambda\omega_\FS)$ by blowups of sizes $\delta_1,\ldots,\delta_k$.
If such a manifold exists, 
then it is unique up to symplectomorphism. 
(This is due to McDuff \cite{isotopy}; we give a more precise statement
in Lemma~\ref{coh implies diff} below.)
With this notation, we now state our main theorem.
For the cases $0 \leq k \leq 2$, see Lemma~\ref{lem:small k}.

\begin{Theorem} \labell{theorem-1}
Let $k \geq 3$.
Given $(M_k,\omega_{\lambda';\delta_1',\ldots,\delta_k'})$,
there exists a unique reduced vector
$(\lambda;\delta_1,\ldots,\delta_k)$
such that
$$ (M_k,\omega_{\lambda';\delta_1',\ldots,\delta_k'})
 \cong (M_k,\omega_{\lambda;\delta_1,\ldots,\delta_k}). $$
\end{Theorem}

\bigskip

To compare different blowups, it is convenient to fix the underlying
manifold $M_k$,  as in~\cite{mystery}.
Once and for all, we fix a sequence 
$$p_1,p_2,p_3,\ldots$$ 
of distinct points
on the complex projective plane $\CP^2$, and we denote by $M_k$
the manifold that is obtained from $\CP^2$ by complex blowups
at $p_1,\ldots,p_k$.  
We have a decomposition
$$ H_2(M_k) = \Z L \oplus \Z E_1 \oplus \ldots \oplus \Z E_k $$
where $L$ is the image of the homology class of a line $\CP^1$
in $\CP^2$ under the inclusion map $H_2(\CP^2) \to H_2(M_k)$
and where $E_1,\ldots,E_k$ are the homology classes of the exceptional
divisors.
A \textbf{blowup form}
on $M_k$ is a symplectic form for which 
there exist pairwise disjoint 
embedded symplectic spheres in the classes $L,E_1,\ldots,E_k$.
(The terminology ``blowup form" was suggested to us by Dusa McDuff.)

The following two lemmas follow from work of Gromov and McDuff. 
Lemma \ref{deformation class}  follows from results 
of Gromov \cite[2.4.A', 2.4.A1']{gromovcurves}, 
McDuff \cite{rational-ruled} and McDuff-Salamon \cite[Proposition~7.21]{intro};
the deduction of Lemma \ref{coh implies diff} from Lemma \ref{deformation class}
is by a result of McDuff \cite{isotopy} using the ``inflation'' technique. 
For details, see \cite{algorithm}. 
For Lemmas~\ref{deformation class} and~\ref{coh implies diff}
in the context of uniqueness questions for symplectic structures,
see \cite[Examples 3.8, 3.9, 3.10]{salamon}.
When $k=0$, Lemma~\ref{coh implies diff}
is Gromov's result~\cite[2.4 $B_2'$ and 2.4 $B_3'$]{gromovcurves},

\begin{Lemma} \labell{deformation class}
The set of blowup forms on $M_k$ is an equivalence class
under the following equivalence relation: 
symplectic forms $\omega$ and $\omega'$ on $M_k$
are equivalent iff there exists a diffeomorphism $f \colon M_k \to M_k$
that acts trivially on the homology
and such that $f^*\omega$ and $\omega'$
are homotopic through symplectic forms.
\end{Lemma}

\begin{Lemma} \labell{coh implies diff}
Any two cohomologous blowup forms on $M_k$
are diffeomorphic through a diffeomorphism that acts trivially 
on the homology.
\end{Lemma}

\begin{Definition} \labell{encode}
Fix a non-negative integer $k$.
Let $\left< \cdot , \cdot \right>$ denote the pairing
between cohomology and homology on $M_k$.
A vector $(\v)$ in $\R^{1+k}$
\textbf{encodes} a cohomology class $\Omega \in H^2(M_k;\R)$
if $\frac{1}{2\pi} \left< \Omega , L \right> = \lambda$
and $\frac{1}{2\pi} \left< \Omega , E_j \right> = \delta_j$
for $j= 1, \ldots, k$.
\end{Definition}

Thus, $\omega_{\v}$ can be taken to be a blowup form on $M_k$
whose cohomology class is encoded by the vector $(\v)$;
by Lemma~\ref{coh implies diff} it is unique
up to a diffeomorphism that acts trivially on the homology.

\begin{Remark} \labell{poslemma}
Suppose that the vector $(\v)$ encodes the cohomology class
of a blowup form $\omega$ on $M_k$.
Then the numbers $\lambda, \delta_1, \ldots, \delta_k$ are positive
(from the definition of ``blowup form"),
they satisfy $\delta_i + \delta_j < \lambda$ for all $i \neq j$
(``the Gromov inequality", see \cite[0.3.B]{gromovcurves}),
and they satisfy 
$\lambda^2 - \delta_1^2 - \ldots - \delta_k^2 > 0$ 
(``the volume inequality").
In particular, if $\delta_1 = \ldots = \delta_k = \lambda/3$, then $k \leq 8$.
\end{Remark}

Theorem \ref{theorem-1}, combined with work of Li-Li \cite{Li-Li02},
further leads to the following characterization of blowup forms,
which we prove in Section~\ref{sec:characterization}.

\begin{Theorem} \labell{theorem-2}
Let $k \geq 3$.  
Let $(\v)$ be a vector with positive entries that is reduced
and that satisfies the volume inequality 
$\lambda^2 - {\delta_1}^2 - \ldots - {\delta_k}^2 > 0$.
Then there exists a blowup form $\omega_{\v}$ whose cohomology
class is encoded by this vector.
This defines
a bijection between the set of vectors 
with positive entries that are reduced and satisfy the volume inequality 
and the set of blowup forms modulo diffeomorphism.
\end{Theorem}

For completeness, we also describe now the cases $0 \leq k \leq 2$,
whose proofs we give in Section~\ref{sec:uniqueness}:

\begin{Lemma} \labell{lem:small k} \ 

\medskip\noindent\textbf{The case $\mathbf{k =2}$: \ }
A vector $(\lambda;\delta_1,\delta_2)$
encodes the cohomology class of a blowup form 
exactly if its entries are positive
and satisfy the Gromov inequality $\delta_1+\delta_2 < \lambda$.
Blowup forms that correspond to vectors
$(\lambda;\delta_1,\delta_2)$ and $(\lambda';\delta'_1,\delta'_2)$
are diffeomorphic
if and only if $\lambda' = \lambda$
and $\{ \delta'_1 , \delta'_2 \} = \{ \delta_1 , \delta_2 \}$.

\medskip\noindent\textbf{The case $\mathbf{k =1}$: \ }
A vector $(\lambda;\delta_1)$ 
encodes the cohomology class of a blowup form
exactly if its entries are positive and satisfy $\delta_1 < \lambda$.
Two blowup forms are diffeomorphic if and only if their cohomology classes
are represented by the same vector.

\medskip\noindent\textbf{The case $\mathbf{k =0}$: \ }
A vector $(\lambda)$ encodes the cohomology class of a blowup form
if and only if $\lambda > 0$.
Two blowup forms are diffeomorphic
if and only if they have the same size $\lambda$.
\end{Lemma}

\bigskip

In Section~\ref{sec:existence} we prove the ``existence" part 
of Theorem~\ref{theorem-1}; see Proposition~\ref{exists reduced}.
In doing this, we give an algorithm that, 
given a vector $v$ that encodes 
the cohomology class of a blowup form,
finds a corresponding reduced vector $v_{\red}$;
see paragraph~\ref{r:algorithm}
and the proof of Proposition~\ref{exists reduced} that follows it.
Moreover, we show that the map $v \mapsto v_{\red}$ 
is continuous; see Lemma~\ref{l:continuous}.

In Section~\ref{sec:minimal},
for every blowup form whose cohomology class is represented 
by a reduced vector,
we give the complete list of exceptional homology classes 
with minimal symplectic area.  
See Theorem \ref{thm:Emin} and Remark~\ref{min E for small k}.
The list always contains the smallest exceptional divisor $E_k$
and generically contains only it.
We give two proofs of this result,
one in Section~\ref{sec:minimal}, 
and one in Section~\ref{sec:mcduff} that uses a beautiful argument of McDuff.

In Section~\ref{sec:uniqueness} we prove the ``uniqueness" part 
of Theorem~\ref{theorem-1}; see Theorem~\ref{uniqueness}.
From this and the results of Section~\ref{sec:existence}
we obtain an algorithm that, given two blowup forms $\omega$ and $\omega'$,
determines whether or not they are diffeomorphic:
apply the algorithm of paragraph~\ref{r:algorithm}
to the vectors $v$ and $v'$ that encode
the cohomology classes $[\omega]$ and $[\omega']$;
the forms $\omega$ and $\omega'$ are diffeomorphic
if and only if $v_{\red} = v'_{\red}$.

In Section~\ref{sec:characterization},
we combine our results with those of Li--Li \cite{Li-Li02}
to obtain Theorem~\ref{theorem-2}.
From this, in turn, we obtain an algorithm that, for every cohomology class,
determines whether or not it contains a blowup form:
apply the algorithm of paragraph~\ref{r:algorithm} 
to the vector $v$ that encodes the cohomology class;
the cohomology class contains a blowup form if and only if 
the entries of $v_{\red}$ are positive.

\bigskip

In this paper, we rely on facts that are rather standard in the 
symplectic topology community
but whose precise statements in the form that we need
are not always explicit in the literature.
More detailed justifications of these statements
are spelled out in an accompanying manuscript \cite{algorithm},
which studies different toric actions on a fixed symplectic four-manifold.

\smallskip

Throughout this paper, unless we say otherwise,
homology is taken with integer coefficients
and cohomology is taken with real coefficients.

\bigskip\noindent\textbf{Acknowledgement.}
This paper branched off from a joint project with Martin Pinsonnault.
We are grateful to Martin for his collaboration.
We are also grateful for stimulating discussions 
with Paul Biran, Dusa McDuff, Dietmar Salamon, and Jake Solomon.  
In particular, 
our communication with McDuff has clarified Theorem~\ref{theorem-2}. 
The second author would also like to acknowledge support from Tamar Ziegler 
of the Technion, Israel Institute of Technology.

\section{Existence of reduced form}
\labell{sec:existence}

In this section 
we prove the ``existence" part of Theorem \ref{theorem-1}:

\begin{Proposition}[Existence of reduced form] \labell{exists reduced}
Let $k \geq 3$.  
Let $\omega$ be a blowup form on $M_k$.
Then there exists a blowup form on $M_k$ that is diffeomorphic to $\omega$
and whose cohomology class is encoded by a reduced vector.
\end{Proposition}

Moreover, in paragraph~\ref{r:algorithm} 
we give an algorithm that associates to every vector $v$ 
that encodes the cohomology class of a blowup form $\omega$
a reduced vector $v_{\red}$ that encodes the cohomology class
of a blowup form that is diffeomorphic to $\omega$,
and in Lemma~\ref{l:continuous} 
we show that the map $v \mapsto v_{\red}$ is continuous.

\bigskip

We begin with some algebraic preliminaries.

We'll consider the $\Z$-module (``the lattice")
with basis elements $L, E_1, \ldots, E_k$:
$$ \Z L \oplus \Z E_1 \oplus \ldots \oplus \Z E_k \quad 
   \left( \cong \Z^{1+k}    \right),$$
with the bilinear form (``the intersection form") that is given by
$$ L \cdot L = 1, \quad E_i \cdot E_i = -1, \quad
   E_i \cdot E_j = 0 \text{ if } i \neq j, \quad
   L \cdot E_j = 0.$$

\begin{noTitle} \labell{identify}
We identify the element $\Omega = (\lambda; \delta_1,\ldots,\delta_k)$
of $\R^{1+k}$ with the homomorphism from the lattice 
$\Z L \oplus \Z E_1 \oplus \ldots \oplus \Z E_k$ to $\R$
that satisfies $\lambda = \frac{1}{2\pi} \Omega(L)$
and $\delta_j = \frac{1}{2\pi} \Omega(E_j)$ for all $1 \leq j \leq k$.
(Of course, we think of each lattice element 
as a homology class in $H_2(M_k)$ and of each vector in $\R^{1+k}$ 
as the cohomology class in $H^2(M_k;\R)$ that it encodes.)
\end{noTitle}

We will use 
the following fact, which we learned from Martin Pinsonnault.
This fact was also a crucial ingredient in our previous work \cite{kkp}.

\begin{Lemma} \labell{no accumulation again}
Let $\Omega := (\v)$ be a vector in $\R^{1+k}$ 
that satisfies the volume inequality
$ \lambda^2 - \delta_1^2 - \ldots - \delta_k^2 > 0$.
Let
$$ \calH_{-1} = \{ E \in \Z L \oplus \Z E_1 \oplus \ldots \oplus \Z E_k
 \ | \ E \cdot E = -1 \}.$$
Then the map $E \mapsto \Omega(E)$ from $\calH_{-1}$ to $\R$ is proper.
That is, for each bounded interval $I \subset \R$,
the set $\{ E \in \calH_{-1} \ | \ \Omega(E) \in I \}$ 
is compact (hence finite).
\end{Lemma}

\begin{proof}
We will refer to the Lorentzian inner product on $\R^{1+k}$:
$$ \langle u, v \rangle = u_0 v_0 - u_1 v_1 - \ldots - u_k v_k $$
for $u = (u_0;u_1,\ldots,u_k)$ and $v = (v_0;v_1,\ldots,v_k)$.
Then 
$\calH_{-1}$ consists of exactly those elements $E$ in the lattice
that have the form
$$ E = aL - b_1 E_1 - \ldots - b_k E_k $$ 
with $u := (a;b_1,\ldots,b_k) \in \Z^{1+k}$
and $\langle u , u \rangle = -1$.
(Thinking of $E$ as a homology class, the vector $u$ encodes its 
Poincar\'e dual.)
For such an $E$, we have
$$ \frac{1}{2\pi} \Omega(E) = \langle \Omega , u \rangle .$$
Because $\Z^{1+k}$ is closed in $\R^{1+k}$, it is enough to show
that the map
$$ u \mapsto \langle \Omega , u \rangle $$
from
$ \calH_{-1}^{\R} := \{ u \in \R^{1+k} \ | \ 
                           \langle u , u \rangle = -1 \} $
to $\R$ is proper.

Recall that $\langle \Omega , \Omega \rangle >0$ (by the volume inequality);
by rescaling, we assume without loss of generality that 
$\langle \Omega , \Omega \rangle = 1$.
Setting $\eps_0 := \Omega$, 
by the Gram-Schmidt procedure there exist 
$\eps_1$, $\ldots$, $\eps_k$ such that
$\langle \eps_0 , \eps_0 \rangle = 1$,
$\langle \eps_j , \eps_j \rangle = -1$ for $1 \leq j \leq k$,
and $\langle \eps_i , \eps_j \rangle = 0$ for $i \neq j$.
In this basis, 
the bilinear form $\left< , \right>$ and hence the set $\calH_{-1}^{\R}$
remain unchanged,
$\Omega$ is represented by the vector $(1;0,\ldots,0)$,
and the map $u \mapsto \langle \Omega , u \rangle$ becomes
$(u_0;u_1,\ldots,u_k) \mapsto u_0$.
It is enough to show that the preimage in $\calH_{-1}^{\R}$
of the interval $[-N,N]$ is compact for each $N > 0$.
This preimage consists of the set of those $(u_0;u_1,\ldots,u_k)$ 
that satisfy the conditions
$u_0^2 - u_1^2 - \ldots - u_k^2 = -1$ and $u_0 \in [-N,N]$.
This set is compact because it is closed and bounded.

\end{proof}

\begin{Definition} \labell{def:cremona}
Let $k \geq 3$.
For any vector $v = (\lambda ; \delta_1, \ldots , \delta_k)$,
define
$$\defect(v) = \delta_1 + \delta_2 + \delta_3 - \lambda,$$
and define the Cremona transformation by 
$$\cremona(v) = (\lambda' ; \delta_1', \ldots , \delta_k'),$$
where
$$\begin{array}{lcl}
\lambda' & = & \lambda - \defect(v) \\
\delta_j' & = &
\begin{cases}
 \delta_j - \defect(v) & \text{ if } 1 \leq j \leq 3 \\
 \delta_j              & \text{ if } 4 \leq j \leq k .
\end{cases}
\end{array} $$
\end{Definition}

\begin{Lemma} \labell{no accumulation 3}
Let $\Omega = (\v)$ be a vector that satisfies the volume inequality
$ \lambda^2 - \delta_1^2 - \ldots \delta_k^2 > 0$.
Then the set of real numbers $\delta_i'$ that occur 
among the last $k$ entries
in vectors $\Omega' = (\lambda';\delta_1',\ldots,\delta_k')$
that can be obtained from $(\v)$ by iterations of the Cremona transformation
(Definition~\ref{def:cremona}) and permutations of the last $k$ entries
has no accumulation points.
\end{Lemma}

\begin{proof}
Identifying $\R^{1+k}$ with the set of homomorphisms from the lattice
$\Z L \oplus \Z E_1 \oplus \ldots \oplus \Z E_k$ to $\R$ 
as in paragraph \ref{identify},
the Cremona transformation of $\R^{1+k}$ is induced by the transformation
of the lattice that is given by 
$$\begin{array}{lll}
 L & \mapsto & 2L - E_1 - E_2 - E_3 \\
 E_1 & \mapsto & L - E_2 - E_3 \\
 E_2 & \mapsto & L - E_3 - E_1 \\
 E_3 & \mapsto & L - E_1 - E_2 \\
 E_j & \mapsto & E_j \quad \text{ if } 4 \leq j \leq k.
\end{array}$$
Similarly, the permutations of $\delta_1,\ldots\delta_k$
are induced from the transformations of the lattice
that preserve $L$ and permute $E_1,\ldots,E_k$.
Thus, if $\Omega' = (\lambda';\delta_1',\ldots,\delta_k')$
is obtained from $\Omega = (\v)$ by iterations
of the Cremona transformation and permutations of the last $k$ entries,
then each $\delta_j' = \frac{1}{2\pi} \Omega'(E_j)$
is equal to $\frac{1}{2\pi} \Omega(E)$ where $E$ is obtained from $E_j$
by the corresponding transformations of the lattice.
Because these transformations preserve the intersection form on the lattice,
we conclude that, for each $j$, the entry $\delta_j'$ belongs to the set
$ \{ \frac{1}{2\pi} \Omega(E) \ | \ E \cdot E = -1 \}$.
By Lemma~\ref{no accumulation again}, this set has no accumulation points.
\end{proof}

\begin{Definition} \labell{standard c move}
Let $k \geq 3$.
The \textbf{standard Cremona move} on $\R^{1+k}$
(cf.\ McDuff and Schlenk \cite{mcsc}) 
is the composition of the following two maps:
\begin{enumerate}
\item[(i)]
The map $(\v) \mapsto (\lambda; \delta_{i_1}, \ldots, \delta_{i_k} )$
that permutes the last $k$ entries 
such that $\delta_{i_1} \geq \ldots \geq \delta_{i_k}$.
\item[(ii)]
The map $v \mapsto \begin{cases}
 \cremona(v) & \text{ if } \defect(v) \geq 0 \\
 v & \text{ otherwise. }
\end{cases}$
\end{enumerate}
\end{Definition}

\begin{Lemma} \labell{std Cremona properties}
\begin{enumerate}
\item
The standard Cremona move is a piecewise linear continuous map
from $\R^{1+k}$ to $\R^{1+k}$.

\item
The standard Cremona move preserves the \textbf{forward positive cone}
$$ \left\{ (\v) \ | \ 
   \lambda^2 - \delta_1^2 - \ldots - \delta_k^2 > 0 \text{ and }  \lambda > 0
\right\}. $$

\item \labell{decreasing}
If $v' = (\lambda';\delta_1',\ldots,\delta_k')$
is obtained from $v = (\v)$ by the standard Cremona move
but is not equal to $v$, then
\begin{enumerate}
\item
$\delta_i' \leq \delta_i$ for all $i$, and for at least one $i$
we have $\delta_i' < \delta_i$; and
\item $\lambda' < \lambda$.
\end{enumerate}

\item
The vectors that are fixed by the standard Cremona move are exactly
the reduced vectors (see Definition~\ref{reduced form}).

\end{enumerate}
\end{Lemma}

We leave the proof of Lemma~\ref{std Cremona properties}
as an exercise to the reader.

\begin{Remark} \labell{not in gp}
Consider the group of transformations of $\R^{1+k}$
that is generated by the Cremona transformation
(Definition~\ref{def:cremona}) and by permutations of the last $k$ entries.
The standard Cremona move is not an element of this group,
but on each vector $v$ it acts through some element of this group
(that depends on $v$).
\end{Remark}

\begin{Lemma} \labell{finitesteps2}
Let $k \geq 3$.
For every vector $v$ in the forward positive cone in $\R^{1+k}$
there exists a positive integer $m$ such that
applying $m$ iterations of the standard Cremona move to $v$
yields a reduced vector in the forward positive cone.
\end{Lemma}

\begin{proof}
Let $v = (\v)$ be a vector in the forward positive cone,
and let $v^{(n)} = (\lambda^{(n)} , \delta_1^{(n)} , \ldots ,  
                                    \delta_k^{(n)} ) $
be the vector that is obtained from $v$ by applying $n$ iterations
of the standard Cremona move.
By Lemma~\ref{std Cremona properties}, for all $n$
\begin{itemize}
\item $\lambda^{(n)} > 0$
\item $(\lambda^{(n)})^2 - (\delta_1^{(n)})^2 - \ldots - (\delta_1^{(n)})^2
                     > 0$
\item $\lambda^{(n)} \leq \lambda$.
\end{itemize}
The second inequality implies 
that $(\delta_i^{(n)})^2 < (\lambda^{(n)})^2$.
The first and third inequalities imply 
that $(\lambda^{(n)})^2 \leq \lambda^2$.
So the numbers $\delta_i^{(n)}$ all lie 
in the bounded interval $(-\lambda,\lambda)$.
By Lemma~\ref{no accumulation 3} and Remark~\ref{not in gp}, 
the set of numbers $\{ \delta_i^{(n)} \}_{1 \leq i \leq k,\ n \in \N} $ 
is finite.  
The third and fourth items of Lemma~\ref{std Cremona properties}
then imply that the sequence of vectors $v^{(n)}$
is eventually constant and hence reduced.
\end{proof}

\begin{Example}
Let $k=6$ and $\frac{1}{3} < \delta < \frac{2}{5}$.
Then the vector 
$(1;\, \delta,\, \delta,\, \delta,\, \delta,\, \delta,\, \delta)$
is not reduced.
Applying the Cremona transformation, we get the vector
$(2-3\delta;\, 
1-2\delta,\, 1-2\delta,\, 1-2\delta,\, \delta,\, \delta,\, \delta)$.
Permuting, we get
$(2-3\delta;\ \delta,\ \delta,\ \delta,\ 1-2\delta,\ 1-2\delta,\ 1-2\delta )$.
Applying the Cremona transformation again, we get
$(4-9\delta;\,
  2-5\delta,\, 2-5\delta,\, 2-5\delta,\, 1-2\delta,\, 1-2\delta,\, 1-2\delta )$;
permuting again, we get
$(4-9\delta;\,
   1-2\delta,\, 1-2\delta,\, 1-2\delta,\, 
   2-5\delta,\, 2-5\delta,\, 2-5\delta )$.
Applying the Cremona transformation a third time, we get
$(5-12\delta;\, 2-5\delta,\, 2-5\delta,\, 2-5\delta,\, 
                2-5\delta,\, 2-5\delta,\, 2-5\delta )$,
which has positive entries and is reduced.
\eoe
\end{Example} 

\begin{Remark} \labell{rk:rational}
In Lemma~\ref{finitesteps2}, if the entries of $v$ are integers, 
then applying iterations of the standard Cremona move
eventually yields a reduced vector by a simpler reason:
then $\lambda^{(n)}$ is a strictly decreasing sequence
of positive integers, so it must be finite.
A similar argument was used in \cite[Proposition 1]{Li}
and again in~\cite[Lemma~3.4]{Li-Li02},
\cite[Lemma~4.7]{Li-Liu}, and~\cite[Prop.~2.3]{ZGQ}.
\end{Remark}

\bigskip

{
We will refer to the genus zero Gromov Witten invariant 
with point constraints,
$$ \GW \colon H_2(M_k) \to \Z .$$
For the precise definition of this invariant, see~\cite{nsmall}.
Fixing a blowup form $\omega$, if $\GW(A) \neq 0$ 
then for generic $\omega$-tamed almost complex structure $J$ 
there exists a $J$-holomorphic sphere in the class $A$.
(We recall that $J$ is $\omega$-tame
if $\omega(u,Ju) > 0$ for all nonzero tangent vectors $u$.)

The Gromov-Witten invariant is the same for all the blowup forms;
this follows from Lemma \ref{deformation class}.
Lemma \ref{deformation class} also implies that the first Chern class
$c_1(TM_k) \in H^2(M_k)$
is the same for all the blowup forms.
Moreover, 
the first Chern class and the Gromov Witten invariant
are consistent under the natural inclusion maps $H_2(M_k) \to H_2(M_{k+1})$
and the natural projection maps  
$H^2(M_{k+1}) \to H^2(M_{k})$;
see \cite[Theorem~1.4]{hu}, \cite[Proposition 3.5]{LP}, 
and the explanation in \cite[Appendix]{algorithm}.

\begin{Lemma}[Characterization of exceptional classes]
\labell{calE and J}
For a homology class $E$ in $H_2(M_k)$, the following conditions are equivalent.
\begin{enumerate}
\item[(a)]
There exists a blowup form $\omega$ such that the class $E$ 
is represented by an embedded $\omega$-symplectic sphere 
with self intersection $-1$.
\item[(b)]
\begin{enumerate}
   \item[(i)]
   $c_1(TM)(E) = 1$,
   \item[(ii)]
$E \cdot E = -1$, and
   \item[(iii)]
$\GW(E) \neq 0$.
\end{enumerate}
\item[(c)]
For every blowup form $\omega$,
the class $E$ is represented by an embedded $\omega$-symplectic sphere 
with self intersection $-1$.
\end{enumerate}
\end{Lemma}

Lemma~\ref{calE and J} follows from  
McDuff's ``$C_1$ lemma" \cite[Lemma 3.1]{rational-ruled}, 
Gromov's compactness theorem \cite[1.5.B]{gromovcurves}, 
and the adjunction formula \cite[Corollary 1.7]{nsmall}.
For some details, see \cite{algorithm}.

\begin{Definition}[Definition of exceptional classes]
\labell{new def exceptional}
A homology class $E$ in $H_2(M_k)$ is \textbf{exceptional} 
if it satisfies the conditions (a), (b), (c) of Lemma \ref{calE and J}.
\end{Definition}

\begin{Remark} [Examples of exceptional classes]
\labell{expexp}
The classes $E_1,\ldots,E_k$ are all exceptional, 
and so are the classes $L-E_{i}-E_{j}$ for all $1\leq i < j \leq k$.
The first fact is by the definition of a blowup form. 
The second fact is since $L-E_{i}-E_{j}$ contains the proper transform 
in the complex blowup $M_k$ of the unique complex line in $\CP^2$ 
through the points $p_i$ and $p_j$; 
this proper transform 
is an embedded complex sphere in $M_k$ 
hence an embedded $\omega$-symplectic sphere with respect to 
a K\"ahler blowup form $\omega$ on $M_k$.
\end{Remark}

}

{
The following lemma is well known. 
It partially follows from Lemma \ref{calE and J} and Remark \ref{expexp}. 
For details, see \cite{algorithm}.

\begin{Lemma} \labell{gwlemma} 
Each of the following homology classes 
has a non-zero Gromov Witten invariant:
$$ L, \quad E_1, \quad \ldots \quad E_k, \quad L-E_i, \quad 
L-E_i-E_j \text{ for $i \neq j$}, \quad \text{ and } \quad
2L - E_1 - E_2 - E_3 .$$
\end{Lemma}
}

\begin{Lemma} \labell{cremona is diffeo 2}
Let $k \geq 3$.
Let $v \in \R^{1+k}$ be a vector in the forward positive cone.
Let $v'$ be the vector that is obtained from $v$
by the standard Cremona move.
Then there exists a blowup form $\omega$ on $M$
whose cohomology class is encoded by $v$
if and only if
there exists a blowup form $\omega'$ on $M$
whose cohomology class is encoded by $v'$.
Moreover, every such $\omega$ and $\omega'$ are diffeomorphic.
\end{Lemma}

\begin{proof}
By Remark~\ref{not in gp}, the vectors $v$ and $v'$ differ
either by the Cremona transformation (Definition~\ref{def:cremona})
or by a transformation that permutes the last $k$ entries.

Identifying $H^2(M_k;\R)$ with $\R^{1+k}$ as in Definition~\ref{encode},
each of these transformations is induced by a diffeomorphism of $M_k$.
Indeed, the Cremona transformation is induced by a diffeomorphism
according to Wall~\cite{wall}.
As for the permutations, they are induced by diffeomorphisms of $M_k$
that are obtained from diffeomorphisms of $\CP^1$ 
that permute the marked points $p_1,\ldots,p_k$
and are biholomorphic on neighbourhoods of these points.

Each of these diffeomorphisms takes $L,E_1,\ldots,E_k$
to homology classes with non-zero Gromov Witten invariants
(see Lemma~\ref{gwlemma});
this implies that these diffeomorphisms
pull back blowup forms to blowup forms.
This and Lemma~\ref{coh implies diff} imply the last part of the result.
\end{proof}

(As we will note in Section \ref{sec:characterization},
by results of Tian-Jun Li, Bang-He Li, and Ai-Ko Liu,
a reduced vector encodes a blowup form
if and only if it is contained in the forward positive cone 
and its entries are positive.)

\begin{noTitle} [\textbf{Algorithm to obtain a reduced form}]
\labell{r:algorithm} 
Let $k \geq 3$. 
Let $v$ be a vector in the forward positive cone in $\R^{1+k}$.
\begin{itemize}
\item[Step 1:]
If $v$ is reduced, declare $v_{\red} = v$ and stop.
\item[Step 2:]
If $v$ is not reduced, replace it by its image under the standard Cremona move
and return to Step 1.
\end{itemize}
By Lemma~\ref{finitesteps2}, 
this algorithm terminates, and it produces a reduced vector $v_{\red}$ 
in the forward positive cone.
Moreover, by Lemma~\ref{cremona is diffeo 2},
if $v$ encodes the cohomology class of a blowup form $\omega$,
then $v_{\red}$
encodes the cohomology class of a blowup form that is diffeomorphic
to $\omega$.
\end{noTitle}

\begin{proof}[Proof of Proposition~\ref{exists reduced}]
The proposition follows immediately from paragraph~\ref{r:algorithm} 
because a vector that encodes the cohomology class of a blowup form
must lie in the forward positive cone.
\end{proof}

\begin{Lemma}\labell{l:continuous}
The function $v \mapsto v_{\red}$ of paragraph~\ref{r:algorithm},
from the forward positive cone to the intersection of the forward positive cone
with the cone of reduced vectors, is continuous.
\end{Lemma}

\begin{noTitle} 
As before,
we consider $\R^{1+k}$ with its Lorentzian inner product
$ \langle u, v \rangle = u_0 v_0 - u_1 v_1 - \ldots - u_k v_k $
for $u = (u_0;u_1,\ldots,u_k)$ and $v = (v_0;v_1,\ldots,v_k)$.
The \emph{null cone}
is the set of $x$ in $\R^{1+k}$ such that $\langle x , x \rangle   = 0$,
the \emph{positive cone}
is the set of $x$ in $\R^{1+k}$ such that $\langle x , x \rangle   > 0$,
and, as already noted, the \emph{forward positive cone}
is the set of $x=(x_0;\ldots,x_k)$ 
such that $\langle x , x \rangle   > 0$ and $x_0 > 0$.

For every nonzero vector $e$, its Lorentzian orthocomplement $e^\perp$
is a hyperplane in $\R^{1+k}$; the hyperplane $e^\perp$ determines 
the vector $e$ up to scalar; every hyperplane is obtained in this way. 
\begin{itemize}
\item
If $\langle e , e \rangle < 0$, 
then the hyperplane $e^\perp$ meets the positive cone,
and the restriction to $e^\perp$ of the Lorentzian metric on $\R^{1+k}$
is also a Lorentzian metric, of type $(1,k-1)$.
\item
If $\langle e , e \rangle > 0$, 
then the hyperplane $e^\perp$ 
does not meet the positive cone, it meets the null cone only at the origin,
and the restriction to $e^\perp$
of the Lorentzian metric on $\R^{1+k}$ is negative definite.
\item
If $\langle e , e \rangle = 0$, 
then the hyperplane $e^\perp$ does not meet the positive cone,
it meets the null cone along the line $\R e$, 
and the restriction to $e^\perp$
of the Lorentzian metric on $\R^{1+k}$ is negative semi-definite
with null space $\R e$. 
\end{itemize}

For a vector $e \in \R^{1+k}$ with $\langle e , e \rangle \neq 0$,
the reflection $\tau_e(v) = v - 2 
  \frac{\langle v , e \rangle }{\langle e , e \rangle} e$
is a Lorentzian isometry that fixes the hyperplane $e^\perp$. 
We call such a map a \textbf{Lorentzian reflection}.
This reflection preserves the forward positive cone
if and only if $\langle e , e \rangle < 0$.
The map $e^\perp \mapsto \tau_e$, for $e$ such that $\langle e , e \rangle < 0$,
embeds the space of Lorentzian hyperplanes (with the topology induced from the 
Grassmannian) into the space of Lorentzian isometries.

The Cremona transformation is the Lorentzian reflection $\tau_e$
that corresponds to the vector $e = (1;1,1,1,0,\ldots,0)$.
The transposition that switches $\delta_i$ and $\delta_{i+1}$
is the Lorentzian reflection $\tau_e$
that corresponds to the vector $e = (0;0,\ldots,-1,1,0,\ldots,0)$
with $\delta_i=-1$, $\delta_{i+1}=1$, and other entries $=0$.
In both of these types of reflections, the vector $e$
has integer entries and satisfies $\langle e , e \rangle = -2$.
\end{noTitle}

The following lemma is a slight reformulation
of an argument of Jake Solomon \cite{solomon}.

\begin{Lemma} \labell{l:jake}
Every compact subset of the forward positive cone in $\R^{1+k}$
meets only finitely many hyperplanes of the form $e^\perp$
for $e \in \Z^{1+k}$ with $\langle e , e \rangle = -2$.
\end{Lemma}

\begin{proof}
If $e$ has integer entries and satisfies $\langle e , e \rangle = -2$, 
then the $(1+k) \times (1+k)$ matrix that represents the reflection $\tau_e$ 
has integer entries. 
Because the set of matrices with integer entries
is a discrete subset of the set of all matrices,
the set of hyperplanes of the form $e^{\perp}$
for $e \in \Z^{1+k}$ with $\langle e , e \rangle = -2$
is discrete in the set of all Lorentzian hyperplanes 
in $\R^{1+k}$.
So a hyperplane that occurs as an accumulation point of such hyperplanes 
(in the Grassmannian) cannot be Lorentzian;
in particular, it cannot meet the forward positive cone.
(In fact, such a hyperplane must be tangent to the null cone.)
The lemma then follows from the compactness of the Grassmannian.
\end{proof}

The hyperplanes of Lemma~\ref{l:jake}
divide the forward positive cone into \emph{chambers}:
the intersections of the forward positive cone
with the closures of the connected components of the complements 
of these hyperplanes.
Note that a Lorentzian isometry takes chambers to chamber.

\begin{Lemma} \labell{chamber2chamber}
The restriction of the standard Cremona move to each chamber
coincides with a Lorentzian isometry that takes 
the chamber to another chamber.
\end{Lemma}

\begin{proof} 
Applying the standard Cremona move to a vector $(\v)$ 
is achieved by iterations of the following single step:
\begin{itemize}
\item
If the vector is reduced, then stop.
\item
Otherwise,
if $\delta_1,\ldots,\delta_k$ are not in weakly decreasing order,
let $i \in \{ 1,\ldots,k-1 \}$ be the smallest index
such that $\delta_{i} < \delta_{i+1}$, and 
switch $\delta_{i}$ and $\delta_{i+1}$.
\item
Otherwise, apply the Cremona transformation (Definition~\ref{def:cremona}).
\end{itemize}

Let $S_0$ denote the cone of reduced vectors.
The hyperplanes that are spanned by its facets
are the fixed point set $\{ \lambda = \delta_1 + \delta_2 + \delta_3 \}$
of the Cremona transformation
and the fixed point set 
$\{ \delta_{i} = \delta_{i+1} \}$ of the transposition
that switches $\delta_i$ and $\delta_{i+1}$ for $i \in \{1,\ldots,k-1\}$.
These hyperplanes divide $\R^{1+k}$ into ``big chambers":
the closures of the connected components of the complements 
of these hyperplanes.
The above single step, restricted to a ``big chamber",
coincides with a Lorentzian reflection.
The lemma then follows from the fact that every chamber
is contained in a ``big chamber".
\end{proof} 

\begin{proof}[Proof of Lemma~\ref{l:continuous}]
Let $x$ be a point in the forward positive cone.
By Lemma~\ref{l:jake}, there exists a neighbourhood $U$ of $x$
that meets only finitely many chambers.
For every chamber there exists a positive integer $m$
such that, on the chamber, the function $v \mapsto v_{\red}$
is obtained by applying $m$ iterations of the standard Cremona move;
this follows from Lemma~\ref{chamber2chamber} and from the fact that
the intersection of the set of reduced vectors 
with the forward positive cone is a union of chambers.
We conclude that there exists a positive integer $m$ 
such that, on the neighbourhood $U$ of $x$,
the function $v \mapsto v_{\red}$
is obtained by applying $m$ iterations of the standard Cremona move.
Because the standard Cremona move is continuous,
the function $v \mapsto v_{\red}$ is continuous near $x$.
Because $x$ was arbitrary,
the function is continuous on the entire forward positive cone.
\end{proof}

\begin{Remark}
We can now give another proof of Lemma~\ref{finitesteps2}.
Let $v$ be a vector in the forward positive cone.
Let $v'$ be a vector in the same chamber as $v$ and whose entries are rational.
By Lemma~\ref{chamber2chamber},
and since the intersection of the set of reduced vectors
with the forward positive cone is a union of chambers,
it is enough to show that there exists a positive integer $m$
such that applying $m$ iterations of the standard Cremona move to $v'$
yields a reduced vector.
This, in turn, follows by applying the argument of Remark~\ref{rk:rational}
to $Nv'$ where $N$ is a positive integer such that $Nv'$ has integer entries.
\end{Remark}

\begin{Remark}
Let $S_m$ denote the set of vectors in $\R^{1+k}$ 
that are brought to reduced form after $m$ or fewer iterations 
of the standard Cremona move 
(but are not necessarily in the forward positive cone).
Let $S$ denote the (increasing) union of the sets $S_m$.
(By Lemma~\ref{finitesteps2}, the forward positive cone is contained in $S$.)

By applying iterations of the standard Cremona move
until we reach a reduced vector, 
we obtain a function $v \mapsto v_{\red}$
that assigns to each vector in the set  $S$
a reduced vector, that is, a vector in $S_0$.
The restriction of this function to each $S_m$,
being the composition of $m$ continuous functions, is continuous.

In particular, $S_0$ is the set of reduced vectors,
given by the linear inequalities~\eqref{conditions-1}.
and the spans of its facets are the fixed point sets
of the Cremona transformation (Definition~\ref{def:cremona})
and of the $k-1$ transpositions of consecutive elements 
$\delta_i , \delta_{i+1}$ for $1 \leq i \leq k-1$.

Consider those cones that can be obtained from $S_0$ by Lorentzian reflections 
through the hyperplanes that are spanned by its facets.
Continue recursively; at each stage we have a collection 
of convex polyhedral cones 
and we add those cones that can be obtained from them 
by Lorentzian reflections through the hyperplanes that are spanned 
by the facets of $S_0$.
The set $S_m$ is a finite union of finite intersections
of such convex polyhedral cones. 
This implies that the union of the interiors of the sets $S_m$
is open and dense in $S$.
Because 
the function $v \mapsto v_{\red}$ is continuous on the interior
of each $S_m$, we conclude that this function is continuous
on an open and dense subset of $S$.
We don't know if this function is continuous on $S$
(or even on the interior of $S$).
\end{Remark}

\begin{Remark}\labell{variation of reduced}
Other authors \cite{Li-Li98,Li,Li-Li02,Li-Liu,ZGQ}
define ``reduced" by the slightly different conditions
$\delta_1 + \delta_2 + \delta_3 \leq \lambda$
and $\delta_1 \geq \ldots \geq \delta_k \geq 0$,
and they allow transformations that flip the signs of the $\delta_i$.
\end{Remark}

\section{Exceptional classes of minimal area}
\labell{sec:minimal}

Let $\omega$ be a blowup form on $M_k$. 
Lemma~\ref{calE and J} and Definition~\ref{new def exceptional}
imply that the set of exceptional classes of minimal $\omega$-area
only depends on the vector $v = (\v)$ 
that encodes the cohomology class $[\omega]$.
We denote this set by 
$$ \calE^v_{\min} .$$
In this section
we identify all the possibilities for the set $\calE^v_{\min}$;
see Theorem~\ref{thm:Emin}, Remark \ref{rk:Emin},
and Remark~\ref{min E for small k}.

More generally, let $\Omega$ be a cohomology class in $H^2(M_k;\R)$
that is encoded by a vector~$v=(\v)$, 
and assume that the set of values 
$\{ \left< \Omega, E \right> \ | \ E \text{ is an exceptional class} \}$
is bounded from below.   
Denote by $ \calE^v_{\min} $ the set of exceptional classes $E$
for which $\frac{1}{2\pi} \left< \Omega, E \right>$ is minimal.
If $v$ satisfies the volume inequality 
$\lambda^2 - \delta_1^2 - \ldots - \delta_k^2 > 0$,
then this set is non-empty and finite, by Lemma~\ref{no accumulation again}.

\bigskip
The following lemma is well known, and is deduced from the 
positivity of intersections of $J$-holomorphic curves 
in four dimensional manifolds 
\cite[Appendix E and Proposition 2.4.4]{nsmall}, 
the Hofer-Lizan-Sikorav regularity criterion \cite{HLS} 
(see also \cite[Lemma 3.3.3]{nsmall}), and the implicit function theorem, 
see \cite[Chapter 3]{nsmall}.  For more details, see \cite{algorithm}.

\begin{Lemma}[Positivity of intersections] 
\labell{positivity of intersections corollary}
Let $A$ and $B$ be homology classes in $H_2(M_k)$
that are linearly independent (over $\R$).
Suppose that $\GW(B) \neq 0$, that $\protect{c_1(TM_k)(A) \geq 1}$, 
and that $A$ is represented by a $J$ holomorphic sphere
for some almost complex structure $J$ that is tamed by some blowup form 
on $M_k$. 
Then the intersection number $A \cdot B$ is non-negative.

In particular, if $E$ is an exceptional class
and $B$ is a class that is not a multiple of~$E$
and with $\GW(B) \neq 0$, then $E \cdot B$ is non-negative.
\end{Lemma}

\bigskip

We recall that
$$ c_1(TM_k) ( L ) = 3 \qquad \text{and} \qquad
   c_1(TM_k) ( E_1 ) = \ldots =
   c_1(TM_k) ( E_k ) = 1.$$

\bigskip

We have the following easy technical lemma.
Suppose $k \geq 3$.
Recall that a vector $(\v)$ with positive entries is \emph{reduced}
if $\delta_1 + \delta_2 + \delta_3 \leq \lambda$
and $\delta_1 \geq \ldots \geq \delta_k$.
Denote
$$
 F := L - E_1 \quad , \quad
 B := L - E_2 \quad , \quad
 E_{12} := L - E_1 - E_2.  $$

\begin{Lemma} \labell{linearlemma}
Let $k \geq 3$.  
Let $\Omega$ be a cohomology class in $H^2(M_k;\R)$ 
that is encoded by a vector $v=(\v)$ 
with positive entries that is reduced.
Let $A$ be a class in $H_2(M_k)$.
Suppose that $A$ is a multiple of one of the classes in the set
\begin{equation} \labell{list of classes}
 \{ L, \ E_1, \ \ldots, \ E_k, \ F, \ B, \ E_{12} \} ,
\end{equation}
and suppose that $c_1(TM_k)(A) \geq 1$.  Then
\begin{equation} \labell{geq deltaklinear}
 \frac{1}{2\pi} \left< \Omega , A \right> \geq \delta_k .
\end{equation}
Moreover, equality in~\eqref{geq deltaklinear} holds
if and only if 
either $A=E_{\ell}$ and $\delta_{\ell}=\delta_k$,
or $A=E_{12}$ and $\lambda-\delta_1-\delta_2=\delta_3=\delta_k$.
\end{Lemma}

In this lemma, $A$ is a homology class over the \emph{integers},
and a-priori it is a \emph{real} multiple
of one of the classes in the set~\eqref{list of classes}.

\begin{proof}
Because $(\v)$ is a reduced vector,
\begin{align}
 & \min_{C \in \{L,E_1,\ldots,E_k,F,B,E_{12}\}}
       \frac{1}{2 \pi}\left<\Omega,C \right> \nonumber \\
 & \ \qquad  \qquad \qquad
    =\min\{\lambda, \ \delta_1, \ \ldots, \ \delta_k, \ 
               \lambda-\delta_1, \ \lambda-\delta_2, \
              \lambda-\delta_1-\delta_2  \} \nonumber \\
 & \ \qquad  \qquad \qquad = \delta_k .  \labell{min in list}
\end{align}
Moreover, the minimum is attained 
on a subset of $\{ E_1, \ldots, E_k, E_{12} \}$
that contains $E_{\ell}$ if and only if $\delta_{\ell} = \delta_k$
and that contains $E_{12}$
if and only if $\lambda - \delta_1 - \delta_2 = \delta_3 = \delta_k$.

Also note that $c_1(TM_k)(C)$ is positive 
for every $C \in \{ L, E_1, \ldots, E_k, F, B, E_{12} \}$.

Each of the sets
$$ \{L,E_1,\ldots,E_k\}, \,\,
   \{F,E_1,\ldots,E_k\}, \,\,  \{F,B,E_{12},E_3,\ldots,E_k\} ,$$
is a basis of $H_2(M_k)$ over $\Z$.
Therefore, the assumption that $A$ is a multiple of a class $C$
in $\{ L, E_1, \ldots, E_k, F, B, E_{12} \} $
and that $c_1(TM_k)(A) > 0$ 
is equivalent to 
$$ A = \gamma C \text{ for an integer } \gamma \geq 1 .$$
The lemma then follows from~\eqref{min in list}.
\end{proof}

\begin{Lemma} \labell{deltak is minimal}
Let $k \geq 3$.  
Let $\Omega$ be a cohomology class in $H^2(M_k;\R)$ 
that is encoded by a vector $v=(\v)$ with positive entries that is reduced.  
Let $A$ be a homology class in $H_2(M_k)$.
Suppose that $\protect{c_1(TM_k)(A) \geq 1}$,
and suppose that $A$ is represented by a $J$-holomorphic sphere
for some almost complex structure $J$
that is tamed by some blowup form on $M_k$.
Then
\begin{equation} \labell{geq deltak}
 \frac{1}{2\pi} \left< \Omega , A \right> \geq \delta_k .
\end{equation}
\end{Lemma}

\begin{Remark}
Equality in~\eqref{geq deltak} implies that $c_1(TM_k)(A)=1$;
we show this in our proof.
We note that, for a class $A$ of a $J$-holomorphic sphere, if $c_1(TM_k)(A)=1$
then either $A$ is an exceptional class or $A \cdot A \geq 0$;
this is a consequence of the adjunction formula.
\end{Remark}

\begin{Remark} \labell{rk:indecomposable}
In \cite{algorithm}, to count toric actions on blowups of $\CP^2$, 
we use the following ``indecomposability of minimal exceptional classes",
which follows from Lemma~\ref{deltak is minimal}:
if $\omega$ is a blowup form 
and $E$ is an exceptional homology class with minimal $\omega$-symplectic area,
then, for every $\omega$-tame almost complex structure $J$,
there exists an embedded $J$-holomorphic sphere in the class $E$.
This result was also obtained by Pinsonnault \cite{pinso}, 
for more general four-manifolds, using Seiberg-Witten-Taubes theory.
\end{Remark}

\begin{Lemma} \labell{deltak is minimal again}
Let $k \geq 3$.  Let $\omega$ be a blowup form 
whose cohomology class is encoded by a vector that is reduced.
Then, for every exceptional class $E$ in $H_2(M_k)$,
$$ \frac{1}{2\pi} \left< [\omega] , E \right> 
   \geq \frac{1}{2\pi} \left< [\omega] , E_k \right>.$$ 
Moreover, let $\Omega$ be a cohomology class in $H^2(M_k;\R)$
that is encoded by a vector $(\v)$ with positive entries that is reduced.
Then, for every exceptional class $E$ in $H_2(M_k)$,
$$ \frac{1}{2\pi} \left< \Omega , E \right> \geq \delta_k .$$
\end{Lemma}

Lemma~\ref{deltak is minimal again} follows from Lemma~\ref{deltak is minimal}.
We prove these lemmas together with the next theorem, 
in which we identify the set $\calE^v_{\min}$ 
of exceptional classes of minimal area.
In the theorem we refer to the following cases.
If $(\v)$ is a vector with positive entries that is reduced,
then exactly one of the following cases occurs,
where
$$ \lambda_F := \lambda - \delta_1, \quad \text{and} \quad
   \delta_{E_{1\ell}} := \lambda - \delta_1 - \delta_\ell
\quad \text{ for } \ell \neq 1.  $$
\begin{enumerate}
\item
$\delta_1 \leq \lambda_F/2$ (equivalently, $\delta_1 \leq \lambda/3$),
 \ and  

\medskip

   \begin{enumerate}
   \item $\delta_k < \lambda/3$.
   \item $\delta_k = \lambda/3$.
   \end{enumerate}

\bigskip

\item
$\delta_1 > \lambda_F/2$, \ $\delta_2 \leq \lambda_F/2$, \ and

\medskip 

   \begin{enumerate}
   \item  $\delta_k < \lambda_F/2$.
   \item  $\delta_k = \lambda_F/2$.  
   \end{enumerate}

\bigskip

\item
$\delta_1 > \lambda_F/2$, \ $\delta_2 > \lambda_F/2$, \ and

\medskip

   \begin{enumerate}
   \item $\delta_k < \delta_{E_{12}}$.
   \item $\delta_k=\delta_{E_{12}}$.
   \end{enumerate}

\end{enumerate}

\begin{Remark} \labell{special v}
Let $k \geq 3$. Let $v = (\v)$ be a vector with positive entries
that is reduced. 
\begin{itemize}
\item
If $v$ is in case (1b),
then $v = (\lambda;\lambda/3,\ldots,\lambda/3)$ and $k \leq 8$.
\item
If $v$ is in case (2b),
then $v = (\lambda;\delta_1,\lambda_F/2,\ldots,\lambda_F/2)$ 
and $\delta_1 > \lambda_F/2$.
\item
If $v$ is in case (3b),
then $v = (\lambda;\delta_1,\delta_2,\delta_{E_{12}},\ldots,\delta_{E_{12}})$
and $\delta_2 > \delta_{E_{12}}$.
\end{itemize}
\end{Remark}

\begin{Theorem}[Exceptional classes with minimal area] \labell{thm:Emin}
Let $k \geq 3$.  
Let $v=(\v)$ be a vector with positive entries that is reduced;
let $\Omega$ be a cohomology class in $H^2(M_k;\R)$ that is encoded
by this vector.
Suppose also that $\Omega$ satisfies the volume inequality
\begin{equation} \labell{volume inequality again}
 \lambda^2 - \delta_1^2 - \dots - \delta_k^2 > 0.
\end{equation}

Let $j$ be the smallest non-negative integer for which
$ \delta_{j+1} = \ldots = \delta_k $.
\smallskip

\begin{itemize}
\item
If $v$ is in one of the cases (1a), (2a), or (3a), then
$$ \calE^v_{\min} = \left\{ E_{j+1} , \ldots , E_k \right\}.$$

\item
If $v$ is in case (1b), then $k\leq 8$, and, 
by Demazure \cite{demazure}, 
the exceptional classes are those classes 
that can be written as $aL -b_1 E_1 - \ldots - b_k E_k$ with 
$(a;b_1,\ldots,b_k)$ a multi-set of one of the following types: 
         $(0;-1,0^{k-1})$, $(1;1^2, 0^{k-2})$,
         $(2;1^5,0^{k-5})$, $(3; 2, 1^6, 0^{k-7})$, 
         $(4;2^3,1^5)$, $(5;2^6, 1^2)$, $(6;3,2^7)$. 
In this case,
$\calE^v_{\min}$ contains \emph{all} the exceptional classes.

\item
If $v$ is in case (2b), 
then 
$$ \calE^v_{\min} = \{ E_2, \ldots, E_k, E_{12}, \ldots, E_{1k} \}.$$

\item
If $v$ is in case (3b), then
$$ \calE^v_{\min} = \{ E_{12}, E_3, \ldots, E_k \} .$$

\end{itemize}

\end{Theorem}

\begin{Remark} \labell{rk:Emin}
Combining Theorem~\ref{thm:Emin}
with the algorithm of Section~\ref{sec:existence},
we obtain the list of exceptional classes 
with minimal $\omega$-area in $H_2(M_k)$
for \emph{any} blowup form $\omega$ on $M_k$, even if its cohomology class
is not represented by a reduced vector.
Indeed, let $v$ be the vector that encodes the cohomology class $[\omega]$.
The algorithm of paragraph~\ref{r:algorithm} 
and Definition~\ref{standard c move} of the standard Cremona move
give maps $\gamma_1,$ $\ldots$, $\gamma_N$ of $\R^{1+k}$
such that each $\gamma_i$ is a permutation of the last $k$ entries
and such that
$v_{\red} := (\cremona \circ \gamma_N \circ \cdots 
 \circ \cremona \circ \gamma_1)(v)$ is reduced.
Let $\ol{\cremona}$, $\ol{\gamma_1}$, $\ldots$, $\ol{\gamma_N}$
be the transformations of $H_2(M_k)$
such that, for every homology class $A$,
identifying every vector $v' \in \R^{1+k}$
with the cohomology class in $H^2(M_k;\R)$ that it encodes, we have 
$\langle \cremona(v') , A \rangle = \langle v', \ol{\cremona}(A) \rangle$
and
$\langle \gamma_i(v') , A \rangle = \langle v', \ol{\gamma_i}(A) \rangle$
for $i=1,\ldots,N$.
Then $\calE_{\min}^v = 
    (\ol{\gamma}_1 \circ \ol{\cremona} \circ \cdots \circ
     \ol{\gamma}_N \circ \ol{\cremona}) \calE_{\min}^{v_{\red}}$.
\end{Remark}

\begin{Remark}[Exceptional classes with minimal area when $k=0,1,2$.]\ 
\labell{min E for small k}

If $k=0$, there are no exceptional classes, so $\calE^v_{\min} = \emptyset$.

If $k=1$, then $\calE^v_{\min} = \{ E_1 \}$.  In fact,
in this case $E_1$ is the only exceptional class.
Indeed, suppose that $E  = aL - b_1 E_1 \in H_2(M_1)$ is exceptional.
The equality $E \cdot E = -1$ gives $a^2 - {b_1} ^2 = -1$,
and the equality $c_1(TM_k)(E)=1$ (see Lemma~\ref{calE and J})
gives $3a - b_1 = 1$.
Because $a$ and $b_1$ are integers, 
we deduce that $a=0$ and $b_1=-1$, i.e., $E=E_1$.

Suppose now that $k=2$.
Then, by Demazure \cite{demazure}, 
the set of exceptional classes is $\{E_1, E_2, E_{12} \}$.
We have $\delta_2 < \lambda_F/2$ exactly if $\delta_2 < \delta_{E_{12}}$
and $\delta_2 = \lambda_F/2$ exactly if $\delta_2 = \delta_{E_{12}}$.

\begin{itemize}
\item 
If $\delta_2 < \lambda_F/2$ and $\delta_2 < \delta_1$,
then $\calE^v_{\min} = \{ E_2 \}$.
\item
If $\delta_2 < \lambda_F/2$ and $\delta_2 = \delta_1$,
then $\calE^v_{\min} = \{ E_1, E_2 \}$.
\item
If $\delta_2 = \lambda_F/2$ and $\delta_2 < \delta_1$,
then $\calE^v_{\min} = \{ E_2, E_{12} \}$.
\item
If $\delta_2 = \lambda_F/2$ and $\delta_2 = \delta_1$,
then $\calE^v_{\min} = \{ E_1, E_2, E_{12} \}$.
\item
If $\delta_2 > \lambda_F/2$,
then $\calE^v_{\min} = \{ E_{12} \}$.
\end{itemize}

\end{Remark}

\begin{proof}[Proof of Lemmas~\ref{deltak is minimal} 
and~\ref{deltak is minimal again} and Theorem \ref{thm:Emin}]

Lemma \ref{deltak is minimal again}
follows from Lemma~\ref{deltak is minimal}:
because $c_1(TM_k)(E) = 1$ and $\GW(E) \neq 0$
(by Lemma~\ref{calE and J} and Definition~\ref{new def exceptional}),
we can apply Lemma~\ref{deltak is minimal} to $E$.

Because $v$ is reduced (see also Remark~\ref{special v}),
in each of the cases in Theorem~\ref{thm:Emin}, 
each of the listed classes is exceptional and has size $\delta_k$.

It remains to prove the following results.
Let $A$ be a homology class in $H_2(M_k)$.
Suppose that $c_1(TM_k)(A) \geq 1$,
and suppose that $A$ is represented by a $J$-holomorphic sphere
for some almost complex structure $J$ that is tamed by some
blowup form on $M_k$.  
Then 
$\frac{1}{2\pi} \left< \Omega , A \right> \geq \delta_k$.
If, moreover, $A$ is an exceptional class with minimal area 
and $v$ satisfies the volume inequality~\eqref{volume inequality again},
then $A$ is one of the classes
that are listed in Theorem~\ref{thm:Emin}, according to the case of $v$.

\smallskip\noindent\textbf{Case 1:
when $\mathbf{\delta_1 \leq \lambda_F/2}$; 
equivalently, $\mathbf{\delta_1 \leq \lambda/3}$. \ }

{
First,
suppose that $A$ is not a multiple of any of the classes $ L,E_1,\ldots,E_k $.
Write
$$A = aL - b_1 E_1 - \ldots - b_k E_k.$$
By Lemma~\ref{gwlemma}, $\GW(L),\GW(E_1),\ldots,\GW(E_k)$ are nonzero;
by Lemma~\ref{positivity of intersections corollary}
and by the assumptions on $A$, the coefficients
$$ a = A \cdot L, \quad b_1 = A \cdot E_1, \quad \ \ldots, \quad
       b_k = A \cdot E_k $$
are nonnegative.
We have
$$ c_1(TM_k)(A) = 3a - b_1 - \ldots - b_k; $$
by assumption, this number is $\geq 1$.
Also in this case, $0 < \delta_i \leq \lambda / 3$ for $i = 1, \ldots, k$.
Thus,
\begin{align}
 \frac{1}{2\pi} \left< \Omega , A \right>
           & = a \lambda - b_1 \delta_1 - \ldots - b_k \delta_k \nonumber \\
 & \geq a \lambda - b_1 \lambda/3 - \ldots - b_k \lambda/3 \nonumber \\
 & = \left( 3a - b_1 - \ldots - b_k \right) \lambda/3 \nonumber \\
 & \stackrel{(\star)}{\geq} \lambda/3 \nonumber \\
 & \geq \delta_1 \nonumber \\
 & \stackrel{(\star\star)}{\geq} \delta_k.      \labell{inequalities1}
\end{align}
(Moreover, equality in~$(\star)$ implies that $c_1(TM_k)(A) = 1$.)

\smallskip

Suppose moreover that $A$ is an exceptional class with minimal area.
The last inequality of~\eqref{inequalities1} being an equality
implies that we are in case (1b).
The class $A$ is then in the set of listed classes because this set
contains \emph{all} the exceptional classes. 

\smallskip

Now, suppose that $A$ is a multiple of one of the classes $ L,E_1,\ldots,E_k $.
Then $\frac{1}{2\pi} \left< \Omega , A \right> \geq \delta_k$,
with equality only if $A$ is one of the classes $ E_{j+1} , \ldots , E_k $,
as in Lemma~\ref{linearlemma}.
These classes are among those that are listed
in all the cases, and in particular in the cases (1a) and (1b).

\medskip

}

\smallskip\noindent\textbf{Case 2:
when $\mathbf{\delta_1 > \lambda_F/2}$
and $\mathbf{\delta_2 \leq \lambda_F/2}$. \ }

First, suppose that $A$ is not a multiple of any of the classes
$ F, E_1, E_2, \ldots, E_k $.
Write
$$A = a_L L + a_F F - b_2 E_2 - \ldots - b_k E_k.$$
By Lemma~\ref{gwlemma},
$\GW(F)$, $\GW(E_1)$, $\GW(E_2)$, $\ldots$, $\GW(E_k)$ are nonzero;
by Lemma~\ref{positivity of intersections corollary} 
and by the assumptions on $A$, the coefficients
$$ a_L = A \cdot F, \quad a_F = A \cdot E_1, \quad
   b_2 = A \cdot E_2, \quad \ldots, \quad b_k = A \cdot E_k $$
are nonnegative.
We have
$$ c_1(TM_k)(A) = 3 a_L + 2 a_F - b_2 - \ldots - b_k ;$$
by assumption, this number is $\geq 1$.
The assumption $\delta_1 >  \lambda_F /2$ 
implies that $\lambda >  \frac{3}{2} \lambda_F$.
Also,
$0 < \delta_j \leq \lambda_F/2$ for all $2 \leq j \leq k$.
Thus,
\begin{align}
\frac{1}{2\pi} \left< \Omega , A \right>
    & = a_L \lambda + a_F \lambda_F - b_2 \delta_2
                 - \ldots - b_k \delta_k             \nonumber \\
 & \geq a_L \frac{3}{2} \lambda_F + a_F \lambda_F
         - b_2 \lambda_F/2 - \ldots - b_k \lambda_F/2 \nonumber \\
 & = \left( 3 a_L + 2 a_F - b_2 - \ldots - b_k \right) \lambda_F/2 \nonumber \\
 & \stackrel{(\star)}{\geq} \lambda_F/2 \nonumber \\
 & \geq \delta_2 \nonumber \\
 & \geq \delta_k. \labell{inequalities2}
\end{align}
(Moreover, equality in~$(\star)$ implies that $c_1(TM_k)(A) = 1$.)

\smallskip

Suppose moreover that $A$ is an exceptional class of minimal area.
The first inequality in~\eqref{inequalities2} being an equality
implies that the coefficient $a_L$ is zero. 
So $-b_2^2 - \ldots - b_k^2 = A \cdot A = -1$
and $2a_F - b_2 - \ldots - b_k = c_1(TM_k)(A) = 1$.
From this we deduce that $A$ is one of the classes $E_{12},\ldots,E_{1k}$.
The last two inequalities of~\eqref{inequalities2} being equalities
implies that we are in case (2b), so $A$ is among the listed classes.

\smallskip

Now suppose that $A$ is a multiple of one of the classes
$ F, E_1, E_2, \ldots, E_k $.
Then $\frac{1}{2\pi} \left< \Omega , A \right> \geq \delta_k$,
with equality only if $A$ is one of the classes
$E_{j+1},\ldots,E_k$, as in Lemma~\ref{linearlemma}.
These classes are among those that are listed
in all the cases, and in particular in the cases (2a) and (2b).

\medskip

\smallskip\noindent\textbf{Case 3:
when $\mathbf{\delta_1 > \lambda_F/2}$
and $\mathbf{\delta_2 > \lambda_F/2}$. \ }

First, suppose that $A$ is not a multiple of any of the classes
$ F, B, E_{12}, E_3, \ldots, E_k $.
Write $$A = a_B B + a_F F - b_{12} E_{12} - b_3 E_3 - \ldots - b_k E_k.$$
By Lemma~\ref{gwlemma},
$\GW(F)$, $\GW(B)$, $\GW(E_{12})$, $\GW(E_1)$,$\ldots$, $\GW(E_k)$ are nonzero;
by Lemma~\ref{positivity of intersections corollary},
the coefficients
$$ a_B = A \cdot F, \quad a_F = A \cdot B, \quad
   b_{12} = A \cdot E_{12}, \quad
   b_3=A \cdot E_3, \quad \ldots, \quad b_k=A \cdot E_k$$
are nonnegative. We have
$$ c_1(TM)(A) = 2 a_B + 2 a_F - b_{12} - b_3 - \ldots - b_k ;$$
by assumption, this number is $\geq 1$.
Let $$\lambda_B = \lambda - \delta_2.$$
The assumption $\delta_1 \geq \delta_2$
implies that $\lambda_B \geq \lambda_F$.
The assumption $\delta_2 >  \lambda_F / 2$
implies that $\delta_{E_{12}} < \lambda_F / 2$.
Because $(\v)$ is reduced, 
$\delta_k \leq \delta_3 \leq \lambda_F - \delta_2$,
which also implies that $0 < \delta_j < \lambda_F/2$ for $j=3,\ldots,k$.
Thus, 
\begin{align}
 \frac{1}{2\pi} \left< \Omega , A \right> & = a_B \lambda_B + a_F \lambda_F
            - b_{12} \delta_{E_{12}} - b_3 \delta_3 - \ldots - b_k \lambda_k
                                                     \nonumber \\
 & \geq a_B \lambda_F + a_F \lambda_F
         - b_{12} \lambda_F/2
         - b_3 \lambda_F/2 - \ldots - b_k \lambda_F/2 \nonumber \\
 & = \left( 2 a_B + 2 a_F - b_{12} - b_3 - \ldots - b_k \right) \lambda_F/2
                                                         \nonumber \\
 & \stackrel{(\star)}{\geq} \lambda_F/2  \nonumber \\
 & > \delta_3 \nonumber\\
 & \geq \delta_k.                         \labell{inequalities3}
\end{align}
(Moreover, equality in~$(\star)$ implies that $c_1(TM_k)(A) = 1$.)

\smallskip

The first inequality in~\eqref{inequalities3} being an equality
implies that $b_{12} = b_3 = \ldots = b_k =0$,
which cannot occur when $A$ is exceptional.

\smallskip

Now, suppose that $A$ is a multiple of one of the classes
$ F, B, E_{12}, E_3, \ldots, E_k $.
Then $\frac{1}{2\pi} \left< \Omega , A \right> \geq \delta_k$.
In case (3a), equality holds only if $A \in \{ E_{j+1} , \ldots , E_k \}$.
In case (3b), equality holds only if $A \in \{ E_{12}, E_{j+1}, \ldots, E_k \}$.
See Lemma~\ref{linearlemma}.
In each of these cases, $A$ belongs to the set of listed classes.

\end{proof}

{
\begin{Corollary} \labell{ABCD}
Let $k \geq 3$.  
Let $\omega$ be a blowup form on $M_k$
whose cohomology class is encoded by a reduced vector $v=(\v)$.
Then one of the following four possibilities (A), (B), (C), (D)
occurs for the set $\calE^v_{\min}$ of exceptional classes 
with minimal area.\smallskip

\begin{itemize}
\item[(A)]
\ \hfill
$\displaystyle{ 
  \calE^v_{\min} \supseteq \{ E_1, E_2, \ldots, E_k, E_{12} \}.}$
\hfill \ \\
In this case, $v = (\lambda, \lambda/3, \ldots, \lambda/3)$. \bigskip

\item[(B)]
\ \hfill
$\displaystyle{
   \calE^v_{\min} = \{ E_2, \ldots, E_k, E_{12}, \ldots, E_{1k} \}.}$
\hfill \ \\
In this case,
$v = (\lambda;\delta_1,\lambda_F/2,\ldots,\lambda_F/2)$,
and $\delta_1 > \lambda_F/2$. \bigskip

\item[(C)]
\ \hfill $\displaystyle{
 \calE^v_{\min} = \{ E_{12}, E_3, \ldots, E_k \} .}$ \hfill \ \\
In this case, $v = (\lambda;\delta_1,\delta_2,\delta_{E_{12}},
\ldots,\delta_{E_{12}})$ and $\delta_2 > \delta_{E_{12}}$. \bigskip

\item[(D)]
\ \hfill $\displaystyle{
 \calE^v_{\min} = \left\{ E_{j+1} , \ldots , E_k \right\},}$ \hfill \ \\
where $j$ is the smallest non-negative integer for which
$\delta_{j+1} = \ldots = \delta_k $.
\end{itemize}
\end{Corollary}
}

\section{McDuff's arguments}
\labell{sec:mcduff}

In Theorem~\ref{thm:Emin} we give the complete list
of exceptional homology classes with minimal symplectic area
for a blowup form whose cohomology class is encoded by a reduced vector.
McDuff has shown us a different proof approach,
which uses the ``reduced" assumption in such a beautiful way 
that we feel compelled to include it.

The following lemma and corollary are a slight variation
of results that were communicated to us by Dusa McDuff.
Their origin  is in  \cite[Lemma 3.4]{McHof}, 
which is attributed to \cite[Lemma 3.4]{Li-Li02}.

\begin{Lemma} \labell{mcdufflem}
Let $k \geq 3$.
Let $A$ be a homology class in $H_2(M_k)$.  Write
$$ A = aL - b_1 E_1 - \ldots - b_k E_k .$$

\begin{enumerate}
\item[(1)]
Suppose that $a \geq 0$ and $ b_{\ell} \geq 0 $ for all $\ell$,
and that
$$ A \cdot A + c_1(TM_k) (A) \geq 0 .$$
Then $0 \leq b_{\ell} \leq a$ for all $\ell$.
If, additionally, 
$$ A \cdot A \geq -1 \quad \text{ and } \quad A \neq 0,$$ 
then there exists $\ell \in \{ 1, \ldots, k \}$
such that $b_\ell < a$.

\item[(2)]
Suppose that $a \geq 0$ and $0 \leq b_{\ell} \leq a$ for all $\ell$,
and that 
$$c_1(TM_k)(A) \geq 0.$$
Let $\Omega$ be a cohomology class in $H^2(M_k;\R)$
that is encoded by a vector $(\v)$ with positive entries
and that is reduced.  Then
$$ \left< \Omega , A \right> \geq 0 .$$
\end{enumerate}
\end{Lemma}

\begin{proof}[Proof of (1)]
Suppose otherwise.
Then $b_{\ell_0} = a + \eta$ for some $\ell_0$ and for some $\eta \geq 1$,
Then 
$$ b_{\ell_0}^2 + b_{\ell_0} = (a+\eta)^2 + (a+\eta)
 = a^2 
   + (\underbrace{ 2\eta+1 }\limits_{\geq 3}) a
   + \underbrace{\eta^2 + \eta}\limits_{> 0}
   > a^2 + 3a,$$ 
and so 
\begin{align*}
 A \cdot A + c_1(TM_k)(A) 
 & = ( a^2 - \sum b_{\ell}^2 ) + ( 3a - \sum b_{\ell} ) \\
 & = \underbrace{( a^2 + 3a ) - ( b_{\ell_0}^2 + b_{\ell_0} )}
     \limits_{<0}
   - \sum_{\ell \neq \ell_0} 
            ( \underbrace{ b_{\ell}^2 + b_{\ell} }\limits_{\geq 0}) \\
 & < 0,
\end{align*}
contradicting our assumption on $A$.
If there does not exist an $\ell$ such that $b_\ell < a$, 
then $b_\ell = a$ for all $\ell$, and $A \cdot A = a^2(1-k)$,
which is $\leq -2$ if $A \neq 0$.
\end{proof}

\begin{proof}[Proof of (2)]
The assumption $c_1(TM_k)(A) \geq 0$ implies that 
$3a - \sum b_{\ell} \geq 0$.
We can then write 
\begin{align} 
\nonumber
   \frac{1}{2\pi} \left< \Omega , A \right>
 &= a\lambda - b_1 \delta_1 - \ldots - b_k \delta_k \\
\labell{wewrite}
 &= \underbrace{\lambda + \ldots + \lambda}\limits_{\text{ $a$ times }}
 - (\underbrace{\delta_1+\ldots+\delta_1}\limits_{\text{ $b_1$ times }}
 + \ldots
 +  \underbrace{\delta_k+\ldots+\delta_k}\limits_{\text{ $b_k$ times }}
 +  \underbrace{0 + \ldots + 0}\limits_{\text{ $3a - \sum b_{\ell}$ times}}).
\end{align}
We set $\delta_{k+1}=0$.

We label the list of $3a$ indices
$$\underbrace{1,\ldots,1}\limits_{\text{ $b_1$ times }}, \quad 
\ldots, \quad
 \underbrace{k,\ldots,k}\limits_{\text{ $b_k$ times }} , \quad
 \underbrace{{k+1},\ldots,{k+1}}\limits_{\text{ $3a - \sum b_{\ell}$ times }} $$
as 
$$ {j_{11}},\ {j_{21}},\ \ldots, {j_{a1}}, \quad 
   {j_{12}},\ {j_{22}},\ \ldots, {j_{a2}}, \quad
   {j_{13}},\ {j_{23}},\ \ldots, {j_{a3}}.$$
Because $0 \leq b_{\ell} \leq a$ for all $\ell$,
for each $1 \leq i \leq a$
those of the three indices $j_{i1},j_{i2},j_{i3}$
that are different from the artificially-added index $k+1$ are distinct.
The right hand side of~\eqref{wewrite} then becomes
\begin{equation} \labell{rhs}
 \sum_{i=1}^a 
   ( \lambda - (\delta_{j_{i1}} + \delta_{j_{i2}} + \delta_{j_{i3}} ) ) 
\end{equation}
where at each summand, $(\delta_{j_{i1}} + \delta_{j_{i2}} + \delta_{j_{i3}})$ is the sum of at most three
of $\delta_1,\ldots,\delta_k$.  Because $(\v)$ is reduced, 
the sum \eqref{rhs} is $\geq 0$.
\end{proof}

\begin{Corollary} \labell{mcduffcor}
Let $k \geq 3$.
Let $\Omega$ be a cohomology class in $H^2(M_k;\R)$
that is encoded by a vector $(\v)$
with positive entries that is reduced.
Let $A$ be a class in $H_2(M_k)$ such that $c_1(TM_k) (A) \geq 1$
and such that $A$ is represented by a $J$ holomorphic sphere
for some almost complex structure $J$
that is tamed by some blowup form on $M_k$.  Then
$$\frac{1}{2\pi} \left< \Omega , A \right> \geq \delta_k.$$

In particular, let $E$ be an exceptional class in $H_2(M_k)$;
then 
$$\frac{1}{2\pi} \left< \Omega , E \right> \geq \delta_k.$$
Moreover,
\begin{itemize}
\item  if $E$ is not one of the classes $E_1,\ldots,E_k$
nor $L - E_1 - E_\ell$ for $\ell \neq 1$, then 
$$\frac{1}{2\pi} \left< \Omega , E \right> \geq \delta_1;$$
\item
if $E$ is not one of the classes $E_1,\ldots,E_k$, then
$$\frac{1}{2\pi} \left< \Omega , E \right> \geq \lambda-\delta_1-\delta_2.$$
\end{itemize}
\end{Corollary}

\bigskip

\begin{proof}
Assume that $A$ is not 
a multiple of any of the classes $L,E_1,\ldots,E_k$;
otherwise, the result is clearly true.
By positivity of intersections 
(Lemma~\ref{positivity of intersections corollary})
we can write
$$ E = aL - b_1 E_1 - \ldots - b_k E_k $$
where $a,b_1,\ldots,b_k$ are nonnegative.
By the adjunction formula,
\begin{align*}
 A \cdot A & \geq c_1(TM_k)(A) - 2  \\
           & \geq -1.
\end{align*}
By part (1) of Lemma~\ref{mcdufflem},
we have $0 \leq b_{\ell} \leq a$ for all $\ell$,
and there exists an $\ell$ such that $b_{\ell} < a$.
We can then apply part (2) of Lemma~\ref{mcdufflem}
to $A - E_\ell$ and conclude that
$$ \frac{1}{2\pi} \left< \Omega , A \right> \geq \delta_\ell .$$

Now, let $E$ be an exceptional class in $H_2(M_k)$
that is not one of the classes $E_1,\ldots,E_k$.
We need to show that
$$ \frac{1}{2\pi} \left< \Omega , E \right> 
    \geq \lambda - \delta_1 - \delta_2 ,$$
and that, if $E$ is not equal to $L - E_1 - E_\ell$
for any $\ell \neq 1$, then
$$ \frac{1}{2\pi} \left< \Omega , E \right> \geq \delta_1.$$

Because $E$ is not one of the classes $E_1,\ldots,E_k$
and is exceptional, $E$ is not a multiple of any of the classes
$L,E_1,\ldots,E_k$,
and by positivity of intersections 
(Lemma~\ref{positivity of intersections corollary})
we can write
$$ E = aL - b_1 E_1 - \ldots - b_k E_k $$
where $a,b_1,\ldots,b_k$ are nonnegative.
By part (1) of Lemma~\ref{mcdufflem},
we have $0 \leq b_{\ell} \leq a$ for all $\ell$.

First, suppose that $b_i = a$ for some $1 \leq i \leq k$.
Then properties $E \cdot E = -1$ and $c_1(TM_k)(E) = 1$
then imply that $E = L - E_i - E_s$ for some $s \neq i$.
Similarly, if $a=1$, then again $E=L-E_{i}-E_{s}$ for $s \neq i$. 
In all these cases
$$ \frac{1}{2\pi} \left< \Omega , E \right> 
   \geq \min_{1\leq i < s \leq k} \{\lambda-\delta_i-\delta_s \}
 = \lambda-\delta_1-\delta_2,$$
and
if $i$ and $s$ are both different from $1$, then 
$$ \frac{1}{2\pi} \left< \Omega , E \right>
 = \lambda - \delta_i - \delta_s \geq \delta_1 $$
because $(\v)$ is reduced.

It remains to consider the case that $a>1$ and $0 \leq b_\ell < a$
for all $\ell$.   Because $E \cdot E=-1$,
there exist two different indices $i, s$ 
such that $b_i>0$ and $b_s>0$. 
We can then apply part~(2) of Lemma~\ref{mcdufflem}
to $ A := E - ( L - E_{i} - E_{s} ) $ and conclude that 
\begin{equation} \labell{geq lambdaEis}
 \frac{1}{2\pi} \left< \Omega , E \right> 
 \geq \lambda - \delta_i - \delta_s .
\end{equation}
Because $(\v)$ is reduced, the right hand side of~\eqref{geq lambdaEis}
is $\geq \lambda - \delta_1 - \delta_2$,
and it is $\geq \delta_1$ if $i$ and $s$ are both different from $1$.
\end{proof} \bigskip 

We now give an alternative proof to Theorem \ref{thm:Emin}, 
using Corollary \ref{mcduffcor}.

\begin{proof}[Proof of Lemmas 
\ref{deltak is minimal} and
\ref{deltak is minimal again}
and Theorem \ref{thm:Emin}]

Lemmas~\ref{deltak is minimal} and~\ref{deltak is minimal again} 
follow from Corollary \ref{mcduffcor}. 

Because $v$ is reduced (see also Remark~\ref{special v}),
in each of the cases in Theorem~\ref{thm:Emin}, 
each of the listed classes is exceptional and has size $\delta_k$.

Now, let $E$ be an exceptional class in $H_2(M_k)$.
By Corollary~\ref{mcduffcor}, $E$ is in $\calE^v_{\min}$
if and only if $\frac{1}{2\pi} \left< \Omega , E \right> = \delta_k$.
If $E$ is one of the classes $E_1,\ldots,E_k$,
then $\frac{1}{2\pi} \left< \Omega , E \right> = \delta_k$
implies that $E$ is in the set $\{ E_{j+1} , \ldots , E_k \}$,
which is contained in all the sets of classes that are listed
in Theorem~\ref{thm:Emin}.  

We now assume that $\frac{1}{2\pi} \left< \Omega , E \right> = \delta_k$
and $E$ is not one of the classes $E_1, \ldots, E_k$.
It remains to prove that $E$ is one of the classes that are listed
in Theorem~\ref{thm:Emin}, according to the case of $v$. 

\smallskip\noindent\textbf{Case 1:
when $\mathbf{\delta_1 \leq \lambda/3}$. \ }

\begin{align}
\frac{1}{2\pi} \left< \Omega , E \right> 
  & \geq    \lambda-\delta_1-\delta_2 
\quad \text{ by Corollary~\ref{mcduffcor} } \nonumber\\
  & \geq \lambda/3 \nonumber\\
  & \geq \delta_1 \nonumber \\
  & \geq \delta_k. \nonumber 
\end{align}
Equality implies that we are in case (1b);
the class $E$ is then in the set of listed classes
because this set contains \emph{all} the exceptional classes.

\smallskip\noindent\textbf{Cases 2 and 3:
when $\mathbf{\delta_1 > \lambda/3}$. \ }

Since $v$ is reduced, we get
 $$ \delta_k \leq \delta_3 
    \leq \frac{\delta_1+ \delta_2+\delta_3}{3}\leq {\lambda}/{3},$$
and so
$$ \delta_k \leq \lambda /3 < \delta_1. $$
By Corollary \ref{mcduffcor}, $E$ is one of the classes $E_{1\ell}$ 
for $\ell > 1$.
It remains to show that this can hold only if $v$ is in case (2b).

Indeed, we now rule out the cases (2a), (3a), and (3b).

{

If $v$ is in case (2a), for all $\ell >1$ we have
\begin{align}
\frac{1}{2\pi} \left< \Omega,E_{1\ell} \right> 
                      &= \lambda-\delta_1-\delta_\ell \nonumber\\
 &\geq \lambda -\delta_1-\delta_2 \nonumber\\
  &= \lambda_F - \delta_2 \nonumber\\
 & \geq \lambda_F/2 \nonumber\\
 &> \delta_k. \nonumber
 \end{align}

If $v$ is case (3a), for all $\ell >1$, we have
\begin{align}
\frac{1}{2\pi}\left< \Omega,E_{1\ell} \right>
 &= \lambda-\delta_1-\delta_\ell \nonumber\\
 &\geq \lambda -\delta_1-\delta_2 \nonumber\\
 &=\delta_{E_{12}} \nonumber\\
 &>\delta_k. \nonumber
\end{align}

If $v$ is in case (3b) then by Remark~\ref{special v}
$\delta_2 > \delta_3 = \ldots = \delta_k
 = \delta_{E_{12}}$, and for $\ell>2$,
\begin{align}
\frac{1}{2\pi} \left< \Omega,E_{1\ell} \right>
 &= \lambda-\delta_1-\delta_\ell \nonumber\\
 & \geq \lambda-\delta_1-\delta_3 \nonumber\\
 & \geq  \delta_2 \nonumber\\
 & > \delta_k. \nonumber
\end{align}

Thus, we have shown that in the cases (2a), (3a), (3b)
the class $E_{1\ell}$ cannot be minimal 
for any $2 \leq \ell \leq k$.
}

\end{proof}

\section{Uniqueness of reduced form}
\labell{sec:uniqueness}

Our goal in this section is to prove the following theorem,
which is the ``uniqueness" part of Theorem~\ref{theorem-1}.

\begin{Theorem} \labell{uniqueness}
Let $k \geq 3$.
Let $\omega$ and $\omega'$ be blowup forms on $M_k$
whose cohomology classes are encoded by the vectors
$$ v = (\lambda;\delta_1,\ldots,\delta_k) 
\qquad \text{ and } \qquad
   v' = (\lambda';\delta_1',\ldots,\delta_k') .$$
Suppose that $v$ and $v'$ are reduced.
Suppose that $(M_k,\omega)$ and $(M_k,\omega')$ are symplectomorphic.
Then $v=v'$. 
\end{Theorem}

Let $(M,\omega)$ be a closed symplectic four-manifold
and $C$ an embedded symplectic sphere of self intersection $-1$.
We recall that a choice of Weinstein tubular neighbourhood of $C$
determines a \emph{symplectic blow-down} $(\ol{M},\ol{\omega})$
of $(M,\omega)$ along $C$, and that we have a natural splitting
\begin{equation} \labell{bdsplit}
  H_2(M) = H_2(\ol{M}) \oplus \Z [C] .  
\end{equation}

We also recall the ``uniqueness of blow downs":
if $C_1$ and $C_2$ are two spheres as above and are in the same homology class, 
and if $(\ol{M}_1,\ol{\omega}_1)$ and $(\ol{M}_2,\ol{\omega}_2)$
are blow-downs of $(M,\omega)$ with respect to some choices of 
Weinstein tubular neighbourhoods of $C_1$ and $C_2$,
then there is a symplectomorphism between 
$(\ol{M}_1,\ol{\omega}_1)$ and $(\ol{M}_2,\ol{\omega}_2)$
that induces the identity map on the second homology
with respect to the decompositions~\eqref{bdsplit}.
An argument for this was given by McDuff in \cite[\S 3]{rational-ruled};
for details, see \cite[Lemma A.1]{ke}.

Finally, suppose that $(\ol{M},\ol{\omega})$ is obtained from $(M,\omega)$
by a symplectic blowdown along a sphere $C$ with respect to some
Weinstein neighbourhood of $C$,
and let $\psi \colon (M,\omega) \to (M',\omega')$ be a symplectomorphism.
Then $\psi$ descends to a symplectomorphism 
from $(\ol{M},\ol{\omega})$ to the manifold $(\ol{M}',\ol{\omega}')$,
that is obtained from $(M',\omega')$
by a symplectic blowdown along $C' := \psi(C)$ with respect to the 
Weinstein tubular neighbourhood that is determined by $\psi$.

\bigskip

\begin{Lemma} \labell{blowdown Ek}
Let $\omega$ be a blowup form on $M_k$.
Let $(\v)$ be the vector that encodes the cohomology class $[\omega]$.
Then there exists an embedded $\omega$-symplectic sphere in the class $E_k$.
For every such a sphere, blowing down along it
yields a symplectic manifold 
that is symplectomorphic to $(M_{k-1},\ol{\omega})$,
where $\ol{\omega}$ is a blowup form, 
and where the cohomology class $[\ol{\omega}]$
is encoded by the vector $(\lambda; \delta_1, \ldots, \delta_{k-1})$.
\end{Lemma}

For details, see \cite{algorithm}.

\bigskip

To proceed, we will need to identify the two-point blowup $M_2$ of $\CP^2$
with the one-point blowup of $S^2 \times S^2$.
We have a decomposition 
$$ H_2(S^2 \times S^2) = \Z B \oplus \Z F $$
where $B = [ S^2 \times \{ \text{point} \}]$ is the ``base class" 
and $F = [ \{ \text{point} \} \times S^2 ]$ is the ``fibre class".
For positive real numbers $a,b$ we consider the split symplectic form
$$ \omega_{a,b} = a \tau_{S^2} \oplus b \tau_{S^2} $$ 
where $\tau_{S^2}$ is the rotation invariant area form on $S^2$,
normalized such that $\protect{\frac{1}{2\pi} \int_{S^2} \tau_{S^2} = 1}$.

\begin{Lemma} \labell{ab}
Suppose that $a \geq b > 0$ and $a' \geq b' > 0$.
Suppose that $(S^2 \times S^2 , \omega_{a,b})$ 
and $(S^2 \times S^2 , \omega_{a',b'})$ are symplectomorphic.
Then $a=a'$ and $b=b'$. 
\end{Lemma}

\begin{proof}
See \cite[Lemma 4]{YK:max_tori}.
\end{proof}

\begin{Lemma} \labell{blowdown E12}
Let $\omega$ be a blowup form on $M_2$.
Then there exists an embedded $\omega$-symplectic sphere 
in the class $E_{12} := L - E_1 - E_2$.
Moreover, for every such a sphere, blowing down along it
yields a symplectic manifold 
that is symplectomorphic to $(S^2 \times S^2 , \omega_{a,b})$,
with $a = \lambda - \delta_2$ and $b = \lambda - \delta_1$,
where $(\lambda;\delta_1,\delta_2)$ is the vector 
that encodes the cohomology class $[\omega]$.
\end{Lemma}

For details, see \cite{algorithm}.

We will also use the following observations on symplectomorphisms 
between blow ups of $\CP^2$.
We say that homology classes are \emph{disjoint} if their intersection
product is zero.

\begin{Lemma} \labell{lemsymp again}
Let $\omega$ and $\omega'$ be blowup forms on $M_k$
whose cohomology classes are encoded by the vectors
$ (\lambda;\delta_1,\ldots,\delta_k) $
and
$ (\lambda';\delta_1',\ldots,\delta_k') $.
Let $\varphi \colon (M_k,\omega) \to (M_k,\omega')$
be a symplectomorphism,
and let $\varphi_* \colon H_2(M_k) \to H_2(M_k)$
be the induced map on the homology.

\begin{enumerate}
\item  The isomorphism $\varphi_*$ preserves the set of exceptional classes.

 \item The isomorphism $\varphi_*$ sends disjoint homology classes
to disjoint homology classes.

\item  The isomorphism $\varphi_*$ 
restricts to a bijection 
from the set of minimal exceptional classes in $(M_k,\omega)$
to the set of minimal exceptional classes in $(M_k,\omega')$,
and 
\begin{equation} \labell{eqphimin again}
 \delta_k = \delta_k'.
\end{equation}

\item 
\begin{equation} \labell{omega times c1}
   3 \lambda - \sum_{i=1}^{k}\delta_i
   = 3 \lambda' - \sum_{i=1}^{k}\delta'_i . 
\end{equation}
\end{enumerate}
\end{Lemma}

\begin{proof}
For details, see~\cite{algorithm}.
We note that~\eqref{eqphimin again} in the case $k \geq 3$
follows from Lemma~\ref{deltak is minimal again}
and that~\eqref{omega times c1} follows from
$\frac{1}{2\pi} \int_{M_k} \omega \wedge c_1(TM_k)
= \frac{1}{2\pi} \int_{M_k} \omega' \wedge c_1(TM_k)$.
\end{proof}

The properties of a symplectomorphism listed in Lemma \ref{lemsymp again} 
and the identification of exceptional classes when $k=1,2$ 
yield the characterization of the blowup forms
when $k \leq 2$ that was stated in Lemma~\ref{lem:small k}.

{
\begin{proof}[Proof of Lemma~\ref{lem:small k}]\ 
The fact that the vector that encodes the cohomology class
of a blowup form satisfies the conditions listed in the lemma
is by definition of a blowup form and by Gromov's packing inequality
\cite[0.3.B]{gromovcurves}.
The fact that the listed conditions on the vector are sufficient
for the cohomology class encoded by the vector to contain a blowup form
can be shown by toric constructions, see e.g., \cite{kk,algorithm}.

By Lemma~\ref{coh implies diff}, 
any two blowup forms whose cohomology classes
are encoded by the same vector are diffeomorphic.
When $k=2$, switching $\delta_1$ and $\delta_2$
can be realized by a diffeomorphism.
It remains to show that if two blowup forms are cohomologous
then the vectors that encode their cohomology classes
are equal or (when $k=2$) differ by switching $\delta_1$ and $\delta_2$.

\medskip\noindent\textbf{Suppose that $\mathbf{k =2}$. \ }
Suppose that there exists a symplectomorphism
from $(M_2,\omega_{\lambda;\delta_1,\delta_2})$
to $(M_2,\omega_{\lambda';\delta'_1,\delta'_2})$.
By Demazure \cite{demazure}, 
the set of exceptional classes in $M_2$ is $\{ E_1, E_2, E_{12} \}$,
and the only pair of disjoint exceptional classes is $\{ E_1, E_2 \}$. 
Because a symplectomorphism takes disjoint exceptional classes
to disjoint exceptional classes, 
$\{ \delta_1, \delta_2 \} = \{ \delta'_1, \delta'_2 \} $.
Because a symplectomorphism preserves the pairing of the symplectic
form with the first Chern class,
$3 \lambda - \delta_1 - \delta_2 = 3 \lambda' - \delta'_1 - \delta'_2$,
which further implies that $\lambda = \lambda'$.
Thus, $(\lambda',\delta'_1,\delta'_2)$
is equal to either $(\lambda,\delta_1,\delta_2)$
or to $(\lambda,\delta_2,\delta_1)$.
Conversely, these two vectors correspond to symplectomorphic manifolds.
For more details, see~\cite{algorithm}. 

\medskip\noindent\textbf{Suppose that $\mathbf{k =1}$. \ }
Suppose that there exists a symplectomorphism
from $(M_2,\omega_{\lambda;\delta_1})$
to $(M_2,\omega_{\lambda';\delta'_1})$. 
As noted in Remark \ref{min E for small k}, in this case $E_1$ is the only exceptional class.
Because a symplectomorphism must take an exceptional class
to an exceptional class,
the symplectomorphism  $(M_k,\omega) \to (M_k,\omega')$
takes the set $\{E_1\}$ to itself.  Thus, $\delta_1=\delta'_1$.
Because a symplectomorphism preserves the pairing of the symplectic
form with the first Chern class,
$3 \lambda - \delta_1 = 3 \lambda' - \delta'_1 $,
which further implies that $\lambda = \lambda'$.

\medskip\noindent\textbf{Suppose that $\mathbf{k =0}$. \ }
On $\CP^2$, 
if two blowup forms are diffeomorphic
then they take the same value on the generator of $H_2(\CP^2)$
on which this value is positive.
So they must have the same size.

\end{proof}
}

\begin{proof}[Proof of Theorem~\ref{uniqueness}] \

Corollary \ref{ABCD} implies that exactly one of the following possibilities 
for the vector $v = (\lambda;\delta_1,\ldots,\delta_k)$ occurs.
A similar list of possibilities holds for the vector
$v' = (\lambda';\delta_1',\ldots,\delta_k')$.

\begin{itemize}

\item[(A)]
Not every two minimal exceptional classes are disjoint,
and there exist $k$ pairwise disjoint minimal exceptional classes.

In this case,
$v = (\lambda;\lambda/3,\ldots,\lambda/3)$.

\item[(B)]
Not every two minimal exceptional classes are disjoint,
and there do not exist $k$ pairwise disjoint minimal exceptional classes.

In this case, 
$v = (\lambda; \delta_1, \lambda_F/2, \ldots, \lambda_F/2)$
and $\delta_1 > \lambda_F/2$.

\item[(C)]
Every two minimal exceptional classes are disjoint, 
and the blowdown of $(M,\omega)$ along all the minimal exceptional classes
yields a manifold that is symplectomorphic 
to $S^2 \times S^2$ with some split symplectic form $\omega_{a,b}$
with $a \geq b > 0$.

In this case, 
$v = (\lambda; \delta_1, \delta_2, \delta_{E_{12}}, \ldots, \delta_{E_{12}} )$,
with $\delta_2 > \delta_{E_{12}}$,
and the parameters $a,b$ are given by 
$a = \lambda - \delta_2$ and $b = \lambda - \delta_1$.

\item[(D)]
Every two minimal exceptional classes are disjoint,
and the blowdown of $(M,\omega)$ along all the minimal exceptional classes
yields a manifold that is symplectomorphic to $(M_j,\ol{\omega})$
for some $0 \leq j < k$, where $\ol{\omega}$ is a blowup class.

In this case, the cohomology class $[\ol{\omega}]$ 
is encoded in the vector $(\lambda;\delta_1,\ldots,\delta_j)$.
\end{itemize}

By items (2) and (3) of Lemma~\ref{lemsymp again},
either $(M_k,\omega)$ and $(M_k,\omega')$
are both in the case (A), or they are both in the case (B),
or they are both in the cases (C) or (D).

In the cases (C) or (D),
because a symplectomorphism between $(M_k,\omega)$ and $(M_k,\omega')$  descends to a symplectomorphism
between the blowdowns along the minimal exceptional spheres,
and because $S^2 \times S^2$ is not symplectomorphic (or even homeomorphic)
to any $M_j$, either both $(M_k,\omega)$ and $(M_k,\omega')$
are in the case (C) or they are both in the case (D).

\bigskip

Suppose $v$ and $v'$ are in case (A).
This means that $v = (\lambda ; {\lambda}/{3} , \ldots , \lambda/3)$
and $v'=(\lambda';\lambda'/3,\ldots,\lambda'/3)$.
Substituting in~\eqref{omega times c1},
the resulting equation implies that $\lambda=\lambda'$, and thus $v=v'$.

Suppose $v$ and $v'$ are in case (B).
This means that
$v = (\lambda ; \delta_1 , {\lambda_{F}}/2 , \ldots , \lambda_F/2)$
and
$v' = (\lambda' ; \delta'_1 , \lambda'_F/2 , \ldots , \lambda'_F/2)$.
Substituting in~\eqref{eqphimin again} and in~\eqref{omega times c1},
and recalling that $\lambda_F = \lambda - \delta_1$ 
and $\lambda'_F = \lambda' - \delta'_1$, 
we get two linearly independent equations 
that imply that $\lambda = \lambda'$ and $\delta_1 = \delta_1'$,
and thus $v=v'$.

Suppose $v$ and $v'$ are in case (C).
Then
$v = (\lambda; \delta_1,\delta_2,\delta_{E_{12}},\ldots,\delta_{E_{12}})$
and 
 $v' = 
   (\lambda'; \delta'_1,\delta'_2,\delta'_{E_{12}},\ldots,\delta'_{E_{12}})$.
By \eqref{eqphimin again}, we get
\begin{equation}  \labell{eqd3}
\delta_{E_{12}} = \delta'_{E_{12}}.
\end{equation}
Because the symplectomorphism descends to a symplectomorphism
between the blowdowns along the minimal exceptional spheres,
and by Lemma~\ref{ab}, we obtain that
$$ \delta_1 + \delta_{E_{12}} = \delta_1' + \delta_{E_{12}}' 
\quad \text{ and } \quad
 \delta_2 + \delta_{E_{12}} = \delta_2' + \delta_{E_{12}}' .$$
By this and~\eqref{eqd3}, we get that 
\begin{equation} \labell{same delta12}
\delta_1 = \delta_1' \quad \text{ and } \delta_2 = \delta_2'.
\end{equation}
Substituting in~\eqref{omega times c1},
we get that $\lambda = \lambda'$.  Thus, $v = v'$.

Suppose $v$ and $v'$ are in case (D).
Because the symplectomorphism descends to a symplectomorphism
between the blowdowns along the minimal exceptional spheres,
we obtain a symplectomorphism between
$(M_j,\ol{\omega})$ and $(M_j,\ol{\omega}')$,
where $[\ol{\omega}]$ is encoded in the vector 
$\ol{v} = (\lambda, \delta_1, \ldots , \delta_j)$ 
and $[\ol{\omega}']$ is encoded in the vector 
$\ol{v}' = (\lambda', \delta_1', \ldots , \delta_j')$. 
Because the vectors $\ol{v}$ and $\ol{v'}$ are reduced,
we can continue by induction.
\end{proof}

\begin{noTitle}
[\textbf{Algorithm to determine whether two blowup forms are diffeomorphic}]
\labell{algorithm 2}
Suppose that $k \geq 3$.
Let $\omega$ and $\omega'$ be blowup forms on $M_k$,
and let $v$ and $v'$ be the vectors that encode 
their cohomology classes.
Apply to each of $v$ and $v'$ the algorithm of paragraph~\ref{r:algorithm}
to obtain reduced vectors $v_{\red}$ and $v'_{\red}$.
Then $\omega$ and $\omega'$ are diffeomorphic
if and only if $v_{\red} = v'_{\red}$.

Indeed, as noted in paragraph~\ref{r:algorithm},
the vectors $v_{\red}$ and $v'_{\red}$
encode cohomology classes of blowup forms $\omega_{\red}$
and $\omega'_{\red}$ that are, respectively, 
diffeomorphic to $\omega$ and to $\omega'$.
If $\omega$ and $\omega'$ are diffeomorphic,
then so are $\omega_{\red}$ and $\omega'_{\red}$,
and, by Theorem~\ref{uniqueness}, we conclude that $v_{\red} = v'_{\red}$.
Conversely, if $v_{\red} = v'_{\red}$,
then $\omega_{\red}$ and $\omega'_{\red}$ are diffeomorphic 
by Lemma~\ref{coh implies diff},
and then $\omega$ and $\omega'$ are diffeomorphic.

If $k=0$, $k=1$, or $k=2$, two blowup forms on $M_k$ are diffeomorphic
if and only if the vectors that encode their cohomology classes are equal
or (in the case $k=2$) differ by switching $\delta_1$ and $\delta_2$.
This follows from Lemma~\ref{lem:small k}.
\end{noTitle}

\begin{Remark}
Zhao, Gao, and Qiu gave another version of ``uniqueness of reduced form" 
\cite{ZGQ}.  They only refer to integral classes.
They work with the slightly different notion of ``reduced form" 
that we described in Remark~\ref{variation of reduced}.
They identify the group that is generated by the relevant
Lorentzian reflections
with the Weyl group of a certain Kac-Moody Lie algebra,
and they rely on properties of such Weyl groups. 
\end{Remark}

\section{Characterization of blowup classes}
\labell{sec:characterization}

In this section we give an algorithm that determines if a cohomology class 
contains a blowup form. The cone of classes of blowup forms on $M_k$ 
is described by Li-Li \cite{Li-Li02} and Li-Liu \cite{Li-Liu}, 
following the work of Biran \cite{biran:packing,biran:divisors}
and McDuff \cite{isotopy}, 
and is explained in McDuff-Schlenk \cite[\S 1.2]{mcsc}. 
We rely on the following two facts.

\begin{enumerate}
\item
Let $\Omega \in H^2(M_k;\R)$ be a cohomology class
that is encoded by a vector in the forward positive cone. 
Then $\Omega$ is the cohomology class of a blowup form on $M_k$
if and only if 
$\left< \Omega , E \right>$ is positive 
for every exceptional class $E$ on $M_k$.

\item
Every exceptional class $E$ on $M_k$ can be obtained from $E_1$
by a sequence of applications of the transformations on $H_2(M_k)$
that induce the Cremona transformation
and the permutations of the $\delta_j$s.
\end{enumerate}

\begin{Remark}
The fact that the cohomology class of every blowup form
satisfies Condition~(1) follows from our definition
of ``exceptional class" (Definition~\ref{new def exceptional} and
Lemma~\ref{calE and J},
which, in turn, relies on Lemma~\ref{deformation class}).

In the works that we quote above,
the authors consider symplectic forms with a standard canonical class,
that is, for which the first Chern class $c_1(TM)$ 
is the same as for blowup forms;
in our notation (Definition \ref{encode}), 
this class is encoded by the vector $(3;1,\ldots,1)$.
And by ``exceptional class",
they refer to a homology class $E$ that is represented by a smoothly embedded
sphere with self intersection $-1$ and such that $c_1(TM)(E) = 1$.
These authors show that a cohomology class $\Omega$ contains a symplectic form 
with standard canonical class if and only if 
it satisfies the two conditions that we listed in (1)
with their interpretation of ``exceptional class".

To use their work, we need to note that 
every homology class that is ``exceptional" in their sense
is also exceptional in our sense, and that every symplectic form
with standard canonical class is a blowup form.

These facts follow from results that are given in 
Part 2 of Lemma 3.5 of \cite{Li-Liu}:
let $\omega$ be a symplectic form with standard canonical class.
\begin{itemize}
\item[--]
If $E$ is an exceptional class in the sense of Li-Li-Liu,
then $E$ is represented by an embedded $\omega$-symplectic sphere.
\item[--]
Every finite set of exceptional classes in the sense of Li-Li-Liu
that are pairwise disjoint (with respect to the intersection form)
is represented by a finite set of embedded $\omega$-symplectic spheres
that are pairwise disjoint (as sets).
\end{itemize}
The first of these results also appeared 
as the ``$-1$ curve theorem" in Theorem A of \cite{liliu-ruled},
which implies that, for every symplectic form on $M$,
if $E$ is an exceptional class in the sense of Li-Li-Liu 
and its pairing with $c_1(TM)$ is positive 
then either $E$ or $-E$ can be represented by an embedded symplectic sphere. 
Li and Liu prove this result using a method of Taubes \cite{taubes-SWG}.

Given a finite set of exceptional classes in the sense of Li-Li-Liu that are
pairwise disjoint with respect to the intersection form,
there exists an $\omega$-tamed almost complex structure $J$ for which 
there exists an
embedded $J$-holomorphic sphere in each of the classes in the set.
This follows from the first result above,
together with 
the Hofer-Lizan-Sikorav regularity criterion \cite{HLS} 
(see also \cite[Lemma 3.3.3]{nsmall}) and the implicit function theorem,
see \cite[Chapter 3]{nsmall}. 
These
spheres are disjoint, as follows from the positivity of intersections of
$J$-holomorphic spheres in four-dimensional manifolds,
see \cite[Appendix E and Proposition 2.4.4]{nsmall},
and the fact that the classes in the given set are pairwise disjoint.
This yields the second result above.

In particular, the classes $E_1,\ldots,E_k$ of the exceptional divisors are
represented by disjoint embedded $\omega$-symplectic spheres. Blowing down
along $k$ disjoint embedded $\omega$-symplectic spheres in the classes
$E_1,\ldots,E_k$ yields a symplectic manifold that is diffeomorphic to
$\CP^2$. By a result of Gromov 
\cite[2.4 $B_2'$ and 2.4 $B_3'$]{gromovcurves}
and a theorem of Taubes, which
uses Seiberg-Witten invariants to guarantee the existence of a
symplectically embedded two-sphere \cite{taubes-SWG-1999},
this resulting manifold is
symplectomorphic to $\CP^2$ with a multiple of the Fubini-Study form 
and $L$ is represented by a symplectically embedded sphere. 
See \cite[Example 3.4]{salamon}. 
We conclude that $\omega$ is a blowup form. Then Lemma \ref{calE and J} and
the first result above show that every exceptional class in the sense of
Li-Li-Liu is also exceptional in our sense. 

\end{Remark}

\begin{Lemma} \labell{reduced is symplectic}
Let $k \geq 3$.
Let $\Omega$ be a cohomology class that is encoded by a vector $(\v)$ 
with positive entries that is reduced.
Suppose that $\Omega$ has positive square.
Then $\Omega$ contains a blowup form.
\end{Lemma}

\begin{proof}
By Lemma~\ref{deltak is minimal again},
for every exceptional class $E$ in $H_2(M_k)$,
we have $\protect{\frac{1}{2\pi} \left< \Omega , E \right> \geq \delta_k}$,
and in particular $\left< \Omega , E \right> > 0$.
The result then follows from the above fact~(1).
\end{proof}

\bigskip

\begin{proof}[Proof of Theorem~\ref{theorem-2}]
Theorem~\ref{theorem-2}\ follows from
Lemma~\ref{finitesteps2}, Lemma~\ref{cremona is diffeo 2},
Lemma~\ref{reduced is symplectic},
and the fact that the cohomology class of any blowup form
is encoded by a vector with positive entries 
that satisfies the volume inequality.
\end{proof}

\bigskip

 \

\begin{noTitle}
[\textbf{Algorithm that, given a cohomology class in $\mathbf{H^2(M_k;\R)}$,
determines whether or not it contains a blowup form}]
\labell{algorithm 3}

The cases $k=0,1,2$ have been addressed in Lemma~\ref{lem:small k}.
Suppose that $k \geq 3$.

Let $v$ denote the vector that encodes the cohomology class.
If $v$ is not in the forward positive cone 
then the cohomology class does not contain any blowup form.
If $v$ is in the forward positive cone, apply the algorithm 
of paragraph~\ref{r:algorithm} to obtain $v_{\red}$. 
If the entries of $v_{\red}$ are all positive, 
then the given cohomology class contains a blowup form.
Otherwise, it does not.

Indeed, by the definition of a blowup form, 
a vector that encodes the cohomology class of a blowup form must be 
in the forward positive cone. As noted in paragraph~\ref{r:algorithm}, 
if $v$ is in the forward positive cone, so is $v_{\red}$ and $v$ encodes 
the cohomology class of some blowup form if and only if $v_{\red}$ does. 
If the entries of $v_{\red}$ are all positive, then 
by Lemma \ref{reduced is symplectic}, the cohomology class encoded 
by $v_{\red}$ contains a blowup form. If the entries of $v_{\red}$ 
are not all positive then, by the definition of a blowup form, 
it cannot encode the cohomology class of a blowup form.
\end{noTitle}

 \

\end{document}